\documentclass[12pt,twoside,reqno]{amsart}
\usepackage{amsmath}
\usepackage{amsfonts}
\usepackage{amssymb}
\usepackage{color}
\usepackage{mathrsfs}
\usepackage{cite}
\usepackage{geometry}
\allowdisplaybreaks
\newcommand{\NN}{\mathbb N} 
\newcommand{\RR}{\mathbb R}

\newcommand{\ml}{\mathcal{L}}

\newcommand{\mb}{\mathcal{B}}

\newcommand{\mi}{\mathcal{I}}
\newcommand{\md}{\mathcal{D}}
\newcommand{\ms}{\mathcal{S}}
\newcommand{\mf}{\mathcal{F}}
\definecolor{cadmiumgreen}{rgb}{0.0, 0.42, 0.24}
\newtheorem{theorem}{Theorem}[section]
\newtheorem{corollary}[theorem]{Corollary}
\newtheorem{lemma}[theorem]{Lemma}
\newtheorem{proposition}[theorem]{Proposition}

\numberwithin{equation}{section}
\begin{document}
\title{Mass threshold  for infinite-time blowup in a chemotaxis model with splitted population} 
%\thanks{}
\author{Philippe Lauren\c{c}ot}
\address{Institut de Math\'ematiques de Toulouse, UMR~5219, Universit\'e de Toulouse, CNRS, F--31062 Toulouse Cedex 9, France}
\email{laurenco@math.univ-toulouse.fr}
\author{Christian Stinner}
\address{Technische Universit\"at Darmstadt, Fachbereich Mathematik, Schlossgartenstr.~7, D--64289 Darmstadt, Germany}
\email{stinner@mathematik.tu-darmstadt.de}
\date{\today}
%
%%%%%%%%%%%%%%%%
%%%%%%%%%%%%%%%%
\begin{abstract}
We study the chemotaxis model 
\begin{align*}
\partial_t u & = \mathrm{div}(\nabla u - u \nabla w) + \theta v -u &\mbox{in } (0,\infty) \times \Omega, \\
\partial_t v & = u - \theta v & \mbox{in } (0,\infty) \times \Omega,\\
\partial_t w & = D \Delta w - \alpha w + v & \mbox{in } (0,\infty) \times \Omega, 
\end{align*}
with no-flux boundary conditions in a bounded and smooth domain $\Omega \subset \RR^2$, where $u$ and $v$ represent the densities of subpopulations of moving 
and static individuals of some species, respectively, and $w$ the concentration of a chemoattractant. We prove that, in an appropriate functional setting, 
all solutions exist globally in time. Moreover, we establish the existence of a critical mass $M_c>0$ of the whole population $u+v$ such that,  
for $M \in (0, M_c)$, any solution is bounded, while, for almost all $M > M_c$, there exist solutions blowing up in infinite time. 
The building block of the analysis is  the construction of a Liapunov functional. As far as we know, this is the first result of this kind when the 
mass conservation includes the two subpopulations and not only the moving one.
\end{abstract}
%%%%%%%%%%%%%%%%
%%%%%%%%%%%%%%%%

\maketitle

%
%     HEADLINES
%
\pagestyle{myheadings}
\markboth{\sc{Ph.~Lauren\c cot \& C.~Stinner}}{Mass threshold in a chemotaxis model}

%%%%%%%%%%%%%%%%
%%%%%%%%%%%%%%%%
\section{Introduction}
%%%%%%%%%%%%%%%%
%%%%%%%%%%%%%%%%

We investigate the dynamics of a chemotaxis model describing the space and time evolution of a species including moving and static individuals, as well as that of a chemoattractant produced by the latter. More precisely, on the one hand, the motion of moving individuals is due to diffusion with a bias towards regions of high concentrations of the chemoattractant. On the other hand, the chemoattractant is produced only by the static individuals, while its spatial fluctuations result from standard diffusion. Finally, the total population in the species is assumed to be constant throughout time evolution, with a linear exchange between the two subpopulations. Denoting the densities of moving and static individuals by $u$ and $v$, respectively, and the concentration of chemoattractant by $w$, the mathematical model reads, after a suitable rescaling of the parameters,
\begin{subequations}\label{aPL}
\begin{align}
\partial_t u & = \mathrm{div}(\nabla u - u \nabla w) + \theta v -u &\mbox{in } (0,\infty) \times \Omega, \label{a1PL} \\
\partial_t v & = u - \theta v & \mbox{in } (0,\infty) \times \Omega, \label{a2PL} \\
\partial_t w & = D \Delta w - \alpha w + v & \mbox{in } (0,\infty) \times \Omega, \label{a3PL}
\end{align}
supplemented with no-flux boundary conditions
\begin{equation}
\nabla u\cdot \mathbf{n} = \nabla w\cdot \mathbf{n} =0 \qquad\mbox{on } (0,\infty) \times \partial \Omega \label{bPL}
\end{equation}
\end{subequations}
and initial conditions
\begin{equation}\label{iPL}
  (u,v,w)(0) = (u_0,v_0,w_0) \qquad\mbox{in } \Omega.
\end{equation}
We assume that $\Omega \subset \mathbb{R}^2$ is a bounded domain with smooth boundary and the constants $D$, $\alpha$, and $\theta$ are positive.

The system \eqref{aPL} can be seen as a particular case or a variant of chemotaxis models derived in \cite{De1977, PMW1999, SSU2016, STP2013, WP1998}, with different interpretations of the species and its two subpopulations. In \cite{De1977}, it is some building material (such as soil) which is, either carried by insects, or deposited on the ground. Proliferation of cancer cells is considered in \cite{SSU2016}, separating migrating cells from proliferating cells, while the spreading  of mountain pine beetles is studied in \cite{PMW1999, STP2013, WP1998}, dividing the population into flying and nesting beetles. 

As far as mathematical analysis is concerned, the model introduced in \cite{De1977}, which also is a simplified variant of the models in \cite{PMW1999, STP2013, WP1998}, is the subject of a number of analytical results dealing with the global existence of solutions and the asymptotic behavior of bounded solutions, see, e.g., \cite{HT2016, Li2020, NNOU2020, TLM2019}. In all these results, instead of the splitting term $\theta v -u$ in \eqref{a1PL}, the corresponding models contain in the taxis equation a dissipative term $f(u)$ depending only on $u$ generalizing the prototype $f(u) = 1 - \mu u$. In case of $f \equiv 0$, a critical mass phenomenon for global solutions is proved in \cite{La2019, TW2017}. For the model developed in \cite{SSU2016}, which includes a splitting term similar to $\theta v -u$  in the taxis equation, the global existence of solutions is proved in \cite{SSU2016}, while further results concerning their large time behavior are lacking.

To the best of our knowledge, our results concerning the behavior of solutions to \eqref{aPL}--\eqref{iPL} seem to be the first going beyond global existence 
for a chemotaxis model involving a species divided into moving and static individuals and containing in the taxis equation a splitting term depending on both subpopulations. 

Our first result states the global existence and well-posedness for \eqref{aPL}--\eqref{iPL} in an appropriate functional setting. To this end, for $r\in (1,\infty)$, we set
\begin{align}\label{in1}
& W^m_{r, \mb}(\Omega) := \left\{ z \in W^m_r (\Omega) \: : \nabla z \cdot \mathbf{n} =0 \mbox{ on } \partial\Omega \right\} \quad\text{if}\quad 1 + \frac{1}{r} < m \le 2, \nonumber \\
& W^m_{r, \mb}(\Omega) := W^m_r (\Omega) \quad\text{if}\quad -1 + \frac{1}{r} < m < 1 + \frac{1}{r}, \\
& W^m_{r, \mb}(\Omega) := W^{-m}_{r/(r-1)} (\Omega)^\prime \quad\text{if}\quad -2 + \frac{1}{r} < m \le -1 + \frac{1}{r}, \nonumber 
\end{align}
and
\begin{equation}\label{in2}
W^m_{r, \mb, +}(\Omega) := \left\{ z \in W^m_{r, \mb} (\Omega) \: : z \ge 0 \mbox{ in } \Omega \right\}, 
\end{equation}
where $W_r^m(\Omega)$, $m\in [0,\infty)$, $r\in [1,\infty)$, denote the usual Sobolev spaces, see \cite[Section~5]{Am1993}. 

%%%%%%%%%%%%%%%%
\begin{theorem}\label{theo1.1}
  Let $M>0$ and $(u_0,v_0,w_0) \in W_{3,+}^1(\Omega;\mathbb{R}^3)$ satisfying 
  	\begin{equation}
  	\|u_0+v_0\|_{L_1(\Omega)} = M. \label{mass0}
  	\end{equation}
  	Then the system \eqref{aPL}--\eqref{iPL} has a unique nonnegative weak solution $(u,v,w)$ in $W^1_3$ defined on $[0,\infty)$ satisfying
  \begin{align*}
    & u \in C \left([0,\infty);W^{1}_{3,+}(\Omega) \right) \cap C^1 \left( [0,\infty); W^1_{3/2} (\Omega ; \RR^2)^\prime \right), \\
    & v \in C^1([0,\infty);W^1_{3,+}(\Omega)), \\
    & w \in C \left([0,\infty);W^{1}_{3,+}(\Omega) \right) \cap C^1 \left( [0,\infty); W^1_{3/2} (\Omega ; \RR^2)^\prime \right),
  \end{align*}
  and
  \begin{equation}\label{mass}
    \| (u+v)(t) \|_{L_1(\Omega)} = M, \qquad t \ge 0.
  \end{equation}
  Moreover,
  \begin{align*}
  & u \in C \left((0,\infty);W^{2}_{3, \mb}(\Omega) \right) \cap C^1 \left( (0,\infty); L_3 (\Omega) \right) \, , \\
  & w \in C \left((0,\infty);W^{2}_{3, \mb}(\Omega) \right) \cap C^1 \left( (0,\infty); L_3 (\Omega) \right) \, .
  \end{align*}
\end{theorem}
%%%%%%%%%%%%%%%%

We next establish a critical mass phenomenon for \eqref{aPL}--\eqref{iPL}. More precisely, we show the existence of a critical mass $M_c >0$,
where $M_c = 4\pi(1+\theta)D$ in the general case and $M_c = 8\pi(1+\theta)D$ in the radial setting in a ball, such that all solutions are bounded 
if the initial mass $M$ satisfies $M < M_c$, while solutions blowing up in infinite time exist for almost all $M > M_c$.

We begin with the statement of the boundedness result for $M$ being subcritical.

%%%%%%%%%%%%%%%%
\begin{theorem}\label{theo1.2}
 Let $M>0$ and consider $(u_0,v_0,w_0) \in \mi_{M}$, where
 \begin{equation}\label{in}
 \mi_{M} := \left\{ (u,v,w) \in W^{1}_{3,+}(\Omega) \times W^1_{3,+}(\Omega) \times W^2_{2, +}(\Omega) \: : \: \| u+v \|_{L_1 (\Omega)} = M \right\}.
 \end{equation}
 By $(u,v,w)$ we denote the solution to \eqref{aPL}--\eqref{iPL}
	given by Theorem~\ref{theo1.1}.
	\begin{enumerate}
	  \item[(a)] If $M = \| u_0+v_0 \|_{L_1(\Omega)} \in (0,4\pi(1+\theta)D)$, then
		  \begin{equation}\label{e1.2.1}
			  \sup\limits_{t \ge 0} \left\{ \| u(t)\|_{L_\infty (\Omega)} + \| v(t)\|_{L_\infty (\Omega)} + \| w(t) \|_{L_\infty (\Omega)} \right\} < \infty.
			\end{equation}
	  \item[(b)] If $\Omega = B_R(0)$ is the ball of radius $R>0$ centered at $x=0$, $(u_0,v_0,w_0)$ are radially symmetric, and 
		  $M = \| u_0+v_0 \|_{L_1(\Omega)} \in (0,8\pi(1+\theta)D)$, then \eqref{e1.2.1} is also satisfied.
	\end{enumerate}
\end{theorem}
%%%%%%%%%%%%%%%%
 
The corresponding unboundedness result for $M$ being supercritical is the following.
 
%%%%%%%%%%%%%%%%
\begin{theorem}\label{theo1.3}
Let $M>0$. 
\begin{enumerate}
 	\item[(a)] If $M\in (4\pi(1+\theta)D,\infty)\setminus (4\pi(1+\theta)D\mathbb{N})$, then there are solutions $(u,v,w)$ to \eqref{aPL}-\eqref{iPL} 
 	with initial conditions in $\mathcal{I}_M$ with an unbounded first component; that is,
\begin{equation}
\lim_{t\to\infty} \|u(t)\|_{L_\infty(\Omega)} = \infty. \label{gu}
\end{equation} 

\item[(b)] Assume that $\Omega=B_R(0)$ for some $R>0$. If $M\in (8\pi(1+\theta)D,\infty)$, then there are solutions $(u,v,w)$ to \eqref{aPL}-\eqref{iPL} with radially symmetric initial conditions in $\mathcal{I}_M$ with an unbounded first component, i.e., satisfying \eqref{gu}. 
\end{enumerate}
\end{theorem}
%%%%%%%%%%%%%%%%

While for Keller-Segel systems there are many results on critical mass phenomena distinguishing between boundedness and finite-time blowup, to the best of our knowledge 
such phenomena separating boundedness from infinite-time blowup seem to be scarcer. Still, the latter phenomenon has been established in \cite{CS2015} for a 
chemotaxis model with volume filling effect as well as in \cite{La2019, TW2017} for models related to \eqref{aPL}, but with $0$ instead of $\theta v - u$ in the right hand side of
\eqref{a1PL}.

We prove the results presented above by mainly extending the strategy from \cite{La2019} to \eqref{aPL}--\eqref{iPL}, which in turn has its roots in \cite{HW2001, Ho2001, Ho2002}. We start by constructing a Liapunov functional for \eqref{aPL}--\eqref{iPL} in Section~\ref{sec2}, see \eqref{dPL}, which is of general interest far beyond the results of this work and actually the building block of our analysis.
In Section~\ref{s3} we prove the global existence of solutions to \eqref{aPL}--\eqref{iPL} in Theorem~\ref{theo1.1} by mainly relying on Amann's theory for partially diffusive parabolic systems in an appropriate functional setting, 
in conjunction with a series of \textit{a priori} estimates, some of them involving parts of the Liapunov functional. In Section~\ref{sec4} we prove that global solutions to \eqref{aPL}--\eqref{iPL} are bounded, provided $M=\|u_0+v_0\|_{L_1(\Omega)}$ is suitably small, see Theorem~\ref{theo1.2}. Here we first prove with the help of the Trudinger-Moser inequality that the Liapunov functional constructed in Section~\ref{sec2} is bounded from below for subcritical $M$ and use this property as a starting point for the derivation of further estimates. Finally, in Section~\ref{sec5} we prove that, for $M$ sufficiently large, unbounded solutions exist, as stated in Theorem~\ref{theo1.3}.
Here we use the strategy pioneered in \cite{Ho2002, HW2001} and further developed in \cite{La2019, SS2001}. On the one hand, we establish that any bounded solution approaches the set of stationary solutions when $t\to\infty$. On the other hand, we show that the Liapunov functional is bounded from below on the set of steady states with fixed mass $M >M_c$, but not bounded from below on the set of initial data with mass $M$. Hence, solutions emanating from initial data for which the Liapunov functional takes sufficiently negative values cannot be global and bounded and therefore 
have to blow up in infinite time. As compared to \cite{La2019}, the Liapunov functional constructed here features additional terms involving $v$, so that some arguments, in particular the proof of blowup, are more involved.

%%%%%%%%%%%%%%%%
%%%%%%%%%%%%%%%%
\section{A Liapunov functional}\label{sec2}
%%%%%%%%%%%%%%%%
%%%%%%%%%%%%%%%%

One of the main contributions of this work is the construction of a Liapunov function for \eqref{aPL}--\eqref{iPL}. To this end, we define for $\theta >0$
\begin{equation}\label{eL}
L(r) := r \ln r -r+1, \qquad  L_\theta (r) := \frac{L(\theta r)}{\theta} = r \ln (\theta r) -r+ \frac{1}{\theta}, \qquad r \ge 0,
\end{equation}
and observe that both functions are nonnegative. We next set
\begin{equation}
\begin{split}
\mathcal{L}(u,v,w) & := \int_\Omega \left( L(u) + L_\theta (v)- (u+v)w \right)\ \mathrm{d}x \\
& \qquad + \frac{1+\theta}{2} \left( D \|\nabla w\|_{L_2 (\Omega)}^2 + \alpha \|w\|_{L_2 (\Omega)}^2 \right) + \frac{1}{2} \|D\Delta w - \alpha w + v\|_{L_2 (\Omega)}^2\ ,
\end{split}\label{dPL}
\end{equation}
and establish that $\mathcal{L}$ is a Liapunov functional for \eqref{aPL}. We emphasize that, in constrast to \cite{La2019}, the Liapunov functional depends not only on $u$ and $w$, but also on $v$ through the term $L_\theta (v)- vw$. This is obviously related to the fact that the conserved quantity throughout time evolution is $\|u+v\|_{L_1(\Omega)}$ instead of $\|u\|_{L_1(\Omega)}$. Thus, some arguments in the forthcoming sections are more involved as compared to \cite{La2019}.

%%%%%%%%%%%%%%%%
\begin{lemma}\label{lem4.2}
Let $M>0$. Consider $(u_0,v_0,w_0)\in \mathcal{I}_M$ and denote the corresponding solution to \eqref{aPL}-\eqref{iPL} given by Theorem~\ref{theo1.1} by $(u,v,w)$. Then
\begin{equation}\label{e4.2.1}
\frac{\mathrm{d}}{\mathrm{d}t} \mathcal{L}(u,v,w) + \mathcal{D}(u,v,w) =0, \qquad t >0,
\end{equation}
where $\md$ is nonnegative and defined in \eqref{ePL} below.
\end{lemma}
%%%%%%%%%%%%%%%%

\begin{proof}
It follows from \eqref{a1PL}, \eqref{a2PL}, and \eqref{bPL} that
\begin{align*}
\frac{\mathrm{d}}{\mathrm{d}t} \int_\Omega \left( u \ln{u} - u - uw \right)\ \mathrm{d}x & = \int_\Omega (\ln u - w) \left[ \mathrm{div}(\nabla u - u \nabla w) + \theta v -u \right]\ \mathrm{d}x \\
& \qquad - \int_\Omega (\partial_t v + \theta v) \partial_t w\ \mathrm{d}x \\
& = - \int_\Omega u |\nabla(\ln{u}-w)|^2\ \mathrm{d}x + \int_\Omega (\theta v -u) \ln{u}\ \mathrm{d}x \\ 
& \qquad + \int_\Omega w \partial_t v\ \mathrm{d}x - \int_\Omega (\partial_t v + \theta v) \partial_t w\ \mathrm{d}x.
\end{align*}
Now, by \eqref{a3PL},
\begin{align*}
\int_\Omega w \partial_t v\ \mathrm{d}x & = \int_\Omega w \left( \partial_t^2 w - D \Delta \partial_t w + \alpha \partial_t w \right)\ \mathrm{d}x \\
& = \int_\Omega \left( \partial_t (w \partial_t w) - (\partial_t w)^2 + D \nabla w\cdot \nabla\partial_t w + \alpha w \partial_t w \right)\ \mathrm{d}x \\
& = \frac{\mathrm{d}}{\mathrm{d}t} \int_\Omega \left(w \partial_t w + \frac{D}{2} |\nabla w|^2 + \frac{\alpha}{2} |w|^2 \right)\ \mathrm{d}x - \|\partial_t w\|_{L_2 (\Omega)}^2 \\
& = \frac{\mathrm{d}}{\mathrm{d}t} \int_\Omega w (D \Delta w-\alpha w + v) \ \mathrm{d}x+  \frac{\mathrm{d}}{\mathrm{d}t} \left( \frac{D}{2} \|\nabla w\|_{L_2 (\Omega)}^2 + \frac{\alpha}{2} \|w\|_{L_2 (\Omega)}^2 \right) - \|\partial_t w\|_{L_2 (\Omega)}^2 \\
& = \frac{\mathrm{d}}{\mathrm{d}t} \left( \int_\Omega v w \ \mathrm{d}x - D \|\nabla w\|_{L_2 (\Omega)}^2 -\alpha \|w\|_{L_2 (\Omega)}^2 + \frac{D}{2} \|\nabla w\|_{L_2 (\Omega)}^2 + \frac{\alpha}{2} \|w\|_{L_2 (\Omega)}^2 \right) \\
& \quad - \|\partial_t w\|_{L_2 (\Omega)}^2 \\
& = - \frac{\mathrm{d}}{\mathrm{d}t} \left( \frac{D}{2} \|\nabla w\|_{L_2 (\Omega)}^2 + \frac{\alpha}{2} \|w\|_{L_2 (\Omega)}^2 - \int_\Omega v w \ \mathrm{d}x \right) - \|\partial_t w\|_{L_2 (\Omega)}^2.
\end{align*}

Using again \eqref{a3PL},
\begin{align*}
- \int_\Omega \partial_t v \partial_t w\ \mathrm{d}x & = - \int_\Omega \partial_t w \left( \partial_t^2 w - D \Delta \partial_t w + \alpha \partial_t w \right)\ \mathrm{d}x \\
& = - \frac{1}{2} \frac{\mathrm{d}}{\mathrm{d}t} \|\partial_t w\|_{L_2 (\Omega)}^2 - D \|\nabla\partial_t w\|_{L_2 (\Omega)}^2 - \alpha \|\partial_t w\|_{L_2 (\Omega)}^2
\end{align*}
and
\begin{align*}
- \theta \int_\Omega v \partial_t w\ \mathrm{d}x & = - \theta \int_\Omega \partial_t w \left( \partial_t w - D \Delta w + \alpha w \right)\ \mathrm{d}x \\ 
& = - \theta \|\partial_t w\|_{L_2 (\Omega)}^2 - \frac{\theta}{2} \frac{\mathrm{d}}{\mathrm{d}t} \left( D \|\nabla w\|_{L_2 (\Omega)}^2 + \alpha \|w\|_{L_2 (\Omega)}^2 \right).
\end{align*}
Gathering the above four identities imply
\begin{align*}
& \frac{\mathrm{d}}{\mathrm{d}t} \int_\Omega \left( u \ln{u} - u -(u+v)w \right)\ \mathrm{d}x + \frac{1+\theta}{2} \frac{\mathrm{d}}{\mathrm{d}t} \left( D \|\nabla w\|_{L_2 (\Omega)}^2 + \alpha \|w\|_{L_2 (\Omega)}^2 \right) 
+ \frac{1}{2} \frac{\mathrm{d}}{\mathrm{d}t} \|\partial_t w\|_{L_2 (\Omega)}^2 \\
& \qquad = - \int_\Omega u |\nabla(\ln{u}-w)|^2\ \mathrm{d}x + \int_\Omega (\theta v - u)\ln{u}\ \mathrm{d}x \\
& \qquad\qquad - D \|\nabla\partial_t w\|_{L_2 (\Omega)}^2 - (1+\theta+\alpha) \|\partial_t w\|_{L_2 (\Omega)}^2.
\end{align*}
Finally, since
\begin{align*}
\frac{\mathrm{d}}{\mathrm{d}t} \int_\Omega (v \ln{(\theta v)} - v)\ \mathrm{d}x & = \int_\Omega \ln{(\theta v)} \partial_t v\ \mathrm{d}x \\
& = \int_\Omega \ln{(\theta v)} (u-\theta v)\ \mathrm{d}x
\end{align*}
by \eqref{a2PL}, we end up with
\begin{equation}
\frac{\mathrm{d}}{\mathrm{d}t} \mathcal{L}(u,v,w) + \mathcal{D}(u,v,w) =0\ , \label{cPL}
\end{equation}
where
\begin{equation}
\begin{split}
\mathcal{D}(u,v,w) & := \int_\Omega u |\nabla(\ln{u}-w)|^2\ \mathrm{d}x + \int_\Omega (\theta v - u)(\ln{(\theta v)} - \ln{u})\ \mathrm{d}x \\
& \qquad + D \|\nabla(D\Delta w - \alpha w + v)\|_{L_2 (\Omega)}^2 + (1+\theta+\alpha) \|D\Delta w - \alpha w + v\|_{L_2 (\Omega)}^2\ .
\end{split} \label{ePL}
\end{equation}
Observe that the monotonicity of the logarithm function ensures the nonnegativity of $\mathcal{D}$.
\end{proof}

A useful consequence of the availability of a Liapunov functional is the stabilization of global solutions which are bounded in a suitable functional space; that is, the cluster points of such solutions as $t\to\infty$ in an appropriate topology are necessarily stationary solutions. In that direction, we report the following result, which is similar to \cite[Theorem~5.2]{GZ1998}, \cite[Lemma~1]{Ho2002}, and \cite[Proposition~3.8]{La2019}.

%%%%%%%%%%%%%%%%
\begin{proposition}\label{prop2.2}
Let $M>0$. Consider $(u_0,v_0,w_0)\in \mathcal{I}_M$ and denote the corresponding solution to \eqref{aPL}-\eqref{iPL} given by Theorem~\ref{theo1.1} by $(u,v,w)$. Assume also that 
\begin{equation}
\sup_{t\ge 0} \|u(t)\|_{L_\infty(\Omega)} < \infty. \label{apbu}
\end{equation}
Then there are a sequence $(t_j)_{j\ge 1}$ of positive times, $t_j\to\infty$, and nonnegative functions $(u_*,v_*,w_*)\in L_\infty(\Omega;\mathbb{R}^2)\times W_{2,\mb}^2(\Omega)$ such that
\begin{subequations}\label{stab}
\begin{equation}
\lim_{j\to\infty} \left[ \|u(t_j)-u_*\|_{L_2(\Omega)} + \|v(t_j)-v_*\|_{L_2(\Omega)} + \|w(t_j)-w_*\|_{W_2^1(\Omega)} \right] = 0 \label{stab1} 
\end{equation}
and
\begin{equation}
\mathcal{L}(u_*,v_*,w_*) \le \liminf_{j\to \infty} \mathcal{L}(u(t_j),v(t_j),w(t_j)), \label{stab2}
\end{equation}
\end{subequations}
where 
\begin{equation*}
u_* = \frac{\theta M}{\theta+1} \frac{e^{w_*}}{\left\| e^{w_*} \right\|_{L_1(\Omega)}}, \qquad v_* = \frac{M}{\theta+1} \frac{e^{w_*}}{\left\| e^{w_*} \right\|_{L_1(\Omega)}},
\end{equation*}
and $w_*$ is a solution to the nonlocal elliptic problem
\begin{equation*}
-D \Delta w_* + \alpha w_* = \frac{M}{\theta+1} \frac{e^{w_*}}{\left\| e^{w_*} \right\|_{L_1(\Omega)}} \;\;\text{ in }\;\; \Omega, \quad \nabla w_*\cdot \mathbf{n} = 0 \;\;\text{ on }\;\; \partial\Omega.
\end{equation*}
\end{proposition}
%%%%%%%%%%%%%%%%

\begin{proof}
The proof is similar to that of \cite[Proposition~3.8]{La2019}, after observing that \eqref{a2PL}, \eqref{apbu}, and the positivity of $\theta$ imply that 
\begin{equation}
\sup_{t\ge 0} \|v(t)\|_{L_\infty(\Omega)} < \infty, \label{apbv}
\end{equation}
while \eqref{a3PL}, \eqref{bPL}, \eqref{mass}, and the positivity of $\alpha$ ensure that
\begin{equation}
\sup_{t\ge 0} \|w(t)\|_{L_1(\Omega)} < \infty, \label{apbw}
\end{equation}
Thanks to \eqref{apbu}, \eqref{apbv}, and \eqref{apbw}, we may argue as in \cite[Proposition~3.8]{La2019} to derive first a lower bound on $\mathcal{L}(u,v,w)$ and then the claimed stabilization.
\end{proof}

%%%%%%%%%%%%%%%%
%%%%%%%%%%%%%%%%
\section{Global existence}\label{s3}
%%%%%%%%%%%%%%%%
%%%%%%%%%%%%%%%%

This section is devoted to the proof of Theorem~\ref{theo1.1} and includes three steps: we first establish the local well-posedness of  \eqref{aPL}--\eqref{iPL} in a suitable functional setting and study the regularity of the solution for positive times. We next derive a series of estimates which excludes the occurrence of finite time blowup.
 
 %%%%%%%%%%%%%%%%
 %%%%%%%%%%%%%%%%
 \subsection{Local well-posedness}\label{s3.1}
 %%%%%%%%%%%%%%%%
 %%%%%%%%%%%%%%%%
 
We start with the local well-posedness of \eqref{aPL}--\eqref{iPL}.

%%%%%%%%%%%%%%%%
\begin{proposition}\label{prop3.1}
  Let  $(u_0,v_0,w_0) \in W^1_3 (\Omega; \RR^3)$ be nonnegative. Then the system \eqref{aPL}--\eqref{iPL} has a unique nonnegative weak solution $(u,v,w)$ in $W^1_3$ defined on a maximal time interval $[0,T_m)$, with $T_m \in (0,\infty]$, satisfying
  \begin{align*}
    (u,w) & \in C \left([0,T_m);W^1_3(\Omega ; \RR^2) \right) \cap C^1 \left( [0,T_m); W^1_{3/2} (\Omega ; \RR^2)^\prime \right), \\
    v & \in C^1([0,T_m);W^1_3(\Omega)),
  \end{align*}
  and
  \begin{equation}\label{mass1}
    \| (u+v)(t) \|_{L_1(\Omega)} = M = \| u_0+v_0 \|_{L_1(\Omega)}, \qquad t \in [0,T_m).
  \end{equation}
  Furthermore, if there is $T>0$ such that
  \begin{equation}\label{e3.1.1}
    (u,v,w) \in BUC \left( [0,T] \cap [0,T_m);W^1_3(\Omega ; \RR^3) \right),
  \end{equation}
  then necessarily $T_m \ge T$.
\end{proposition}
%%%%%%%%%%%%%%%%

\begin{proof}
Throughout the proof, $C$ denotes a positive constant that may vary from line to line and depends only on $\Omega$, $\theta$, $D$, $\alpha$, $\|u_0\|_{W_3^1(\Omega)}$, $\|v_0\|_{W_3^1(\Omega)}$, and $\|w_0\|_{W_3^1(\Omega)}$. The existence and uniqueness of a weak solution $(u,v,w)$ to \eqref{aPL}--\eqref{iPL} having the claimed properties, except for nonnegativity and \eqref{mass1}, are proved in a completely similar way as \cite[Proposition~2.1]{La2019}. The proof relies on the framework developed in \cite{Am1991a,Am1993} to handle systems coupling parabolic equations and ordinary differential equations. Specifically, we set $U = (u_1,u_2,u_3) := (u,w,v)$, $U^1 := (u,w)$, $U^2 := v$, and define
\begin{equation*}
A(V) := \begin{pmatrix} 1 & -v_1 \\ 0 & D \end{pmatrix}, \qquad
  S^1(V) := \begin{pmatrix} \theta v_3 -v_1 \\ v_3 - \alpha v_2 \end{pmatrix}, \qquad 
  S^2(V) := ( v_1 - \theta v_3)
  \end{equation*}
  for $V=(v_1,v_2,v_3)$. Introducing the operators
  \begin{align*}
  & \mathcal{A}^1(V)U^1 := - \sum_{j=1}^2 \partial_j \left( A(V) \partial_j U^1 \right), \quad \mathcal{B}^1(V)U^1 := \sum_{j=1}^2 A(V) \mathbf{n}\cdot \nabla U^1, \\
  & \mathcal{A}^2(V)U^2 := 0 \ , \quad \mathcal{B}^2(V)U^2 := 0, 
  \end{align*}
  the system \eqref{aPL}--\eqref{iPL} can be recast as
  \begin{align*}
  \partial_t U^1 + \mathcal{A}^1(U)U^1 + \mathcal{A}^2(U)U^2 & = S^1(U) \;\;\text{ in }\;\; (0,\infty)\times\Omega, \\
  \partial_t U^2 & = S^2(U) \;\;\text{ in }\;\; (0,\infty)\times\Omega, \\
  \mathcal{B}^1(U)U^1 + \mathcal{B}^2(U)U^2 & = 0 \;\;\text{ on }\;\; (0,\infty)\times\partial\Omega, \\
  U(0) & = (u_0,w_0,v_0) \;\;\text{ in }\;\; \Omega,
  \end{align*}
  and its well-posedness, as stated in Proposition~\ref{prop3.1}, follows from \cite[Theorem~6.4]{Am1991a}. Indeed, this result can be applied here since $(\mathcal{A}^1,\mathcal{B}^1)$ is normally elliptic by \cite[Remarks~4.1~(a)-(iii)]{Am1991a} and \cite[Condition~(6.1)]{Am1991a} is satisfied. Denoting the solution to the above system provided by \cite[Theorem~6.4]{Am1991a} by $U=(U^1,U^2)$, we set $(u,w) := U^1$ and $v:=U^2$ and thereby obtain Proposition~\ref{prop3.1}, but yet without the nonnegativity of $(u,v,w)$ and the mass conservation \eqref{mass1}.
  
  To prove the former, we first notice that, since $W_{3}^{11/6}(\Omega)$ embeds continuously in $W_\infty^1(\Omega)$ by \cite[Chapter~7, Theorem~1.2]{Pa1983} and
  \begin{align*}
  \left( L_3(\Omega) , W_{3,\mathcal{B}}^2(\Omega) \right)_{11/12,3} & \doteq W_{3,\mathcal{B}}^{11/6}(\Omega), \\
  \left( L_3(\Omega) , W_{3,\mathcal{B}}^2(\Omega) \right)_{11/24,3} & \doteq W_{3,\mathcal{B}}^{11/12}(\Omega),
  \end{align*}
  by \cite[Eq.~(5.3) and Theorem~5.2]{Am1993}, we infer from \eqref{a3PL} and regularizing properties of the semigroup in $L_3(\Omega)$ associated with the unbounded operator $-D \Delta + \alpha\mathrm{id}$ with domain $W_{3,\mathcal{B}}^2(\Omega)$ that, for $t\in [0,T_m)$,
  \begin{align*}
 \|w(t)\|_{W_\infty^1(\Omega)} & \le C \|w(t)\|_{W_3^{11/6}(\Omega)} \\ 
 & \le C e^{-\alpha t/2} (Dt)^{-11/24} \|w_0\|_{W_3^{11/12}(\Omega)} + C \int_0^t e^{-\alpha(t-s)/2} (D(t-s))^{-11/24} \|v(s)\|_{W_3^{11/12}(\Omega)} \ \mathrm{d}s\\
 & \le C t^{-11/24} + C \left( \sup_{s\in [0,t]}\left\{ \|v(s)\|_{W_3^1(\Omega)}\right\} \right) t^{13/24},
  \end{align*}
  where we have used the continuous embedding of $W_3^1(\Omega)$ in $W_3^{11/12}(\Omega)$ to obtain the last inequality. Since $v\in C([0,T_m);W_3^1(\Omega)$, the above inequality implies that, for $T\in (0,T_m)$,
  \begin{align*}
  \int_0^T \|w(t)\|_{W_\infty^1(\Omega)}^2\ \mathrm{d}t \le C T^{1/12} + C \left( \sup_{s\in [0,t]}\left\{ \|v(s)\|_{W_3^1(\Omega)}\right\} \right)^2 T^{25/12};
  \end{align*}
  that is, 
  \begin{equation}
  w \in L_2((0,T);W_\infty^1(\Omega))\ , \qquad T\in (0,T_m)\ . \label{u1}
  \end{equation}
  Let us now recall that the positive part $r_+$ of a real number is given by $r_+ := \max\{ r , 0\}$ and set $N(r) := (-r)_+$ for $r\in\mathbb{R}$. Then, 
  \begin{equation}
  N'(r)^2 = - N'(r)\ , \quad r N'(r) = N(r)\ , \quad (N N')(r) = - N(r)\ , \quad r N(r) = - N(r)^2\ . \label{u2}
  \end{equation}
  On the one hand, it follows from \eqref{a1PL}, \eqref{u2}, and Young's inequality that 
  \begin{align*}
  \frac{1}{2} \frac{\mathrm{d}}{\mathrm{d}t} \|N(u)\|_{L_2(\Omega)}^2 & = \int_\Omega N'(u) |\nabla u|^2\ \mathrm{d}x - \int_\Omega u N'(u) \nabla u \cdot \nabla w\ \mathrm{d}x \\ 
  & \qquad - \theta \int_\Omega N(u) (v_+ - N(v))\ \mathrm{d}x - \|N(u)\|_{L_2(\Omega)}^2 \\
  & \le - \|\nabla N(u)\|_{L_2(\Omega)}^2 + \int_\Omega u N'(u)^2 \nabla u \cdot \nabla w\ \mathrm{d}x \\
  & \qquad + \theta \int_\Omega N(u) N(v)\ \mathrm{d}x - \|N(u)\|_{L_2(\Omega)}^2 \\
  & \le - \|\nabla N(u)\|_{L_2(\Omega)}^2 + \int_\Omega N(u) \nabla N(u) \cdot \nabla w\ \mathrm{d}x \\
  & \qquad + \theta \int_\Omega N(u) N(v)\ \mathrm{d}x - \|N(u)\|_{L_2(\Omega)}^2 \\
  & \le - \frac{1}{2} \|\nabla N(u)\|_{L_2(\Omega)}^2 + \frac{1}{2} \|N(u)\|_{L_2(\Omega)}^2 \|\nabla w\|_{L_\infty(\Omega)}^2 \\
  & \qquad + \theta \int_\Omega N(u) N(v)\ \mathrm{d}x - \|N(u)\|_{L_2(\Omega)}^2\ . 
  \end{align*}
  On the other hand, we infer from \eqref{a2PL} and \eqref{u2} that
  \begin{align*}
  \frac{\theta}{2} \frac{\mathrm{d}}{\mathrm{d}t} \|N(v)\|_{L_2(\Omega)}^2 & = - \theta \int_\Omega N(v)  (u_+ -N(u) - \theta v)\ \mathrm{d}x \\
  & \le \theta \int_\Omega N(v) N(u)\ \mathrm{d}x - \theta^2 \|N(v)\|_{L_2(\Omega}^2\ .
  \end{align*}
  Summing up the previous two differential inequalities leads us to
  \begin{align*}
  \frac{\mathrm{d}}{\mathrm{d}t} \left( \|N(u)\|_{L_2(\Omega)}^2 + \theta \|N(v)\|_{L_2(\Omega)}^2 \right) & \le \|N(u)\|_{L_2(\Omega)}^2 \|w\|_{W_\infty^1(\Omega)}^2 - 2 \|N(u) - \theta N(v)\|_{L_2(\Omega)}^2 \\
  & \le \|w\|_{W_\infty^1(\Omega)}^2 \left( \|N(u)\|_{L_2(\Omega)}^2 + \theta \|N(v)\|_{L_2(\Omega)}^2 \right)\ .
\end{align*}  
Hence, by \eqref{u1} and the nonnegativity of $u_0$ and $v_0$,
\begin{align*}
& \|N(u(t))\|_{L_2(\Omega)}^2 + \theta \|N(v(t))\|_{L_2(\Omega)}^2 \\
& \qquad \le \left( \|N(u_0)\|_{L_2(\Omega)}^2 + \theta \|N(v_0)\|_{L_2(\Omega)}^2 \right) \exp\left\{ \int_0^t \|w(s)\|_{W_\infty^1(\Omega)}^2\ \mathrm{d}s \right\} = 0
\end{align*}
for $t\in [0,T_m)$, from which the nonnegativity of $u$ and $v$ in $[0,T_m)$ follows. That $w$ is also nonnegative in $[0,T_m)$ is next a straightforward consequence of that of $v$ and the comparison principle applied to \eqref{a3PL}.
  
We finally integrate \eqref{a1PL} and \eqref{a2PL} over $\Omega$ and deduce from the no-flux boundary conditions \eqref{bPL} that 
\begin{equation*}
\frac{\mathrm{d}}{\mathrm{d}t} \int_\Omega (u+v)\ \mathrm{d}x = \int_\Omega  \mathrm{div}(\nabla u - u \nabla w)\ \mathrm{d}x =0\ , \qquad t \in [0,T_m),
\end{equation*}
from which \eqref{mass1} follows by the nonnegativity of $u$ and $v$.
\end{proof}

%%%%%%%%%%%%%%%%
%%%%%%%%%%%%%%%%
\subsection{Smoothness for positive times}\label{s3.2}
%%%%%%%%%%%%%%%%
%%%%%%%%%%%%%%%%

Owing to the regularizing properties of the Laplace operator and the associated semigroup in $L_3(\Omega)$, the solution to \eqref{aPL}--\eqref{iPL} constructed in Proposition~\ref{prop3.1} is more regular for positive times, as reported in the next result.
 
%%%%%%%%%%%%%%%%
\begin{corollary}\label{cor3.2}
Let the assumptions from Proposition~\ref{prop3.1} be fulfilled and denote the weak solution to \eqref{aPL}--\eqref{iPL} constructed in Proposition~\ref{prop3.1} by $(u,v,w)$. Then
  \begin{align}
     & u \in C\left((0,T_m);W^{2}_{3, \mb}(\Omega) \right) \cap C^1 \left( (0,T_m); L_3 (\Omega) \right), \label{e3.2.3}  \\
     & w \in C \left((0,T_m);W^{2}_{3, \mb}(\Omega) \right) \cap C^1 \left( (0,T_m); L_3 (\Omega) \right). \label{e3.2.4}
  \end{align}
\end{corollary}
%%%%%%%%%%%%%%%%

\begin{proof}
We adapt and refine the proof of \cite[Corollary~2.2]{La2019}. Let $\beta\in (0,1)$. Since $v\in C^1([0,T_m);W^{1}_{3}(\Omega))$ and $-D \Delta + \alpha\, \mathrm{id}$ generates an analytic semigroup in $L_3(\Omega)$, we infer from \eqref{a3PL} and \cite[Theorem~10.1]{Am1993} (with $\rho=\beta$, $E_0=L_3(\Omega)$, and $E_1=W^{2}_{3,\mathcal{B}}(\Omega)$) that 
\begin{equation}
w \in C^\beta((0,T_m);W^{2}_{3,\mathcal{B}}(\Omega)) \cap C^{1+\beta}((0,T_m);L_3(\Omega)). \label{et1}
\end{equation}

Next, fix $\gamma\in (5/6,1)$. On the one hand, by \eqref{in1} and Proposition~\ref{prop3.1}, 
\begin{equation}
u\in C([0,T_m);W_{3,\mathcal{B}}^1(\Omega)) \cap C^1([0,T_m);W_{3,\mathcal{B}}^{-1}(\Omega)), \label{et0}
\end{equation}
while \cite[Theorem~7.2]{Am1993} guarantees that	
\begin{equation*}
\left( W_{3,\mathcal{B}}^{-1}(\Omega) , W_{3,\mathcal{B}}^{1}(\Omega) \right)_{\gamma,3} \doteq W_{3,\mathcal{B}}^{2\gamma-1}(\Omega)
\end{equation*}
(up to equivalent norms). Then, for $0\le s\le t<T_m$,
\begin{align*}
\|u(t)-u(s)\|_{W_{3,\mb}^{2\gamma-1}(\Omega)} & \le \|u(t)-u(s)\|_{W_{3,\mb}^{-1}(\Omega)}^{1-\gamma} \|u(t)-u(s)\|_{W_{3,\mb}^{1}(\Omega)}^\gamma \\
& \le 2^\gamma (t-s)^{1-\gamma} \sup_{[s,t]} \|\partial_t u\|_{W_{3,\mb}^{-1}(\Omega)}^{1-\gamma} \sup_{[s,t]} \|u\|_{W_{3,\mb}^{1}(\Omega)}^\gamma,
\end{align*}
and we deduce from \eqref{et0} that
\begin{equation}
u \in C^{1-\gamma}([0,T_m);W_{3,\mathcal{B}}^{2\gamma-1}(\Omega)). \label{et2}
\end{equation}
On the other hand, using \eqref{et1} with $\beta=1-\gamma$ along with the continuous embedding of $W_{3,\mb}^2(\Omega)$ in $W_{3,\mb}^{2\gamma}(\Omega)$, we obtain that
\begin{equation}
w\in C^{1-\gamma}((0,T_m);W_{3,\mb}^{2\gamma}(\Omega)). \label{et3}
\end{equation}
Owing to the choice of $\gamma$ which ensures that $2\gamma-1>2/3$, the space $W_{3,\mathcal{B}}^{2\gamma-1}(\Omega)$ is an algebra and we deduce from \eqref{et2} and \eqref{et3} that $u \nabla w\in C^{1-\gamma}((0,T_m);W_{3,\mathcal{B}}^{2\gamma-1}(\Omega;\mathbb{R}^2))$. Thus,
\begin{equation}
\mathrm{div}(u\nabla w) \in C^{1-\gamma}((0,T_m);W_{3,\mathcal{B}}^{2\gamma-2}(\Omega))\ . \label{et4}
\end{equation}
Now, let $\varepsilon\in (0,T_m)$. Thanks to \cite[Theorem~8.5]{Am1993}, the realization in $W_{3,\mathcal{B}}^{2\gamma-2}(\Omega)$ of the Laplace operator with homogeneous Neumann boundary conditions generates an analytic semigroup in $W_{3,\mathcal{B}}^{2\gamma-2}(\Omega)$, its domain being $W_{3,\mathcal{B}}^{2\gamma}(\Omega)$, and it follows from \eqref{a1PL}, \eqref{et4}, and \cite[Theorem~10.1]{Am1993} (with $\rho=1-\gamma$, $E_0=W_{3,\mathcal{B}}^{2\gamma-2}(\Omega)$, and $E_1=W_{3,\mathcal{B}}^{2\gamma}(\Omega)$) that
\begin{equation}
u \in C^{1-\gamma}((\varepsilon/2,T_m);W_{3,\mathcal{B}}^{2\gamma}(\Omega)) \cap C^{2-\gamma}((\varepsilon/2,T_m);W_{3,\mathcal{B}}^{2\gamma-2}(\Omega))\ . \label{et5}
\end{equation}

Finally, fix $\eta\in (5/3-\gamma,1)$ and notice that \cite[Theorem~7.2]{Am1993} guarantees that, up to equivalent norms,
\begin{equation*}
\left( W_{3,\mathcal{B}}^{2\gamma-2}(\Omega) , W_{3,\mathcal{B}}^{2\gamma}(\Omega) \right)_{\eta,3} \doteq W_{3,\mathcal{B}}^{2(\eta+\gamma)-2}(\Omega).
\end{equation*}
Therefore, for $\varepsilon/2< s\le t<T_m$,
\begin{align*}
\|u(t)-u(s)\|_{W_{3,\mb}^{2(\eta+\gamma)-2}(\Omega)} & \le \|u(t)-u(s)\|_{W_{3,\mb}^{2\gamma-2}(\Omega)}^{1-\eta} \|u(t)-u(s)\|_{W_{3,\mb}^{2\gamma}(\Omega)}^\eta \\
& \le (t-s)^{1-\eta} \sup_{[s,t]} \|\partial_t u\|_{W_{3,\mb}^{2\gamma-2}(\Omega)}^{1-\eta} (t-s)^{\eta(1-\gamma)}  \|u\|_{C^{1-\gamma}([s,t];W_{3,\mb}^{1}(\Omega))}^\eta,
\end{align*}
and we infer from \eqref{et5}, that 
\begin{equation*}
u \in C^{1-\gamma\eta}((\varepsilon/2,T_m); W_{3,\mathcal{B}}^{2(\eta+\gamma)-2}(\Omega)).
\end{equation*}
Hence, since $2(\eta+\gamma)-2>4/3>1$, 
\begin{equation}
u \in C^{1-\gamma\eta}((\varepsilon/2,T_m); W_{3,\mathcal{B}}^{1}(\Omega))\ . \label{et6}
\end{equation}
Combining \eqref{et1} (with $\beta=1-\gamma\eta$) and \eqref{et6} and recalling that $W_{3,\mathcal{B}}^1(\Omega)$ is an algebra entail that $u\nabla w$ belongs to $C^{1-\gamma\eta}((\varepsilon/2,T_m);W_{3,\mathcal{B}}^1(\Omega;\mathbb{R}^2))$ and thus
\begin{equation}
\mathrm{div}(u\nabla w) \in C^{1-\gamma\eta}((\varepsilon/2,T_m);L_3(\Omega)). \label{et7}
\end{equation}
In view of \eqref{a1PL} and \eqref{et7}, another application of \cite[Theorem~10.1]{Am1993} (with $\rho=1-\gamma\eta$, $E_0=L_3(\Omega)$, $E_1=W_{3,\mathcal{B}}^2(\Omega)$) gives
\begin{equation*}
u \in C^{1-\gamma\eta}((\varepsilon,T_m);W_{3,\mathcal{B}}^2(\Omega))\cap C^{2-\gamma\eta}((\varepsilon,T_m);L_3(\Omega)).
\end{equation*}
Since $\varepsilon\in (0,T_m)$ is arbitrary, the proof of Corollary~\ref{cor3.2} is complete. 
\end{proof}

%%%%%%%%%%%%%%%%
%%%%%%%%%%%%%%%%
\subsection{Estimates and global existence}\label{s3.3}
%%%%%%%%%%%%%%%%
%%%%%%%%%%%%%%%%

We now prove the global existence for \eqref{aPL}--\eqref{iPL} and aim at showing that the solution $(u,v,w)$ from Proposition~\ref{prop3.1} satisfies \eqref{e3.1.1} for all $T>0$. To this end, we take advantage of the outcome of Corollary~\ref{cor3.2} which guarantees higher regularity for $(u(t),v(t),w(t))$ for $t\in (0,T_m)$ and derive estimates on $[t_0,T_m)$ for some fixed but arbitrary $t_0\in (0,T_m)$.

Let us thus fix $t_0\in (0,T_m)$ and recall that Corollary~\ref{cor3.2} ensures that
\begin{equation}
(u(t_0),w(t_0))\in W_{3,\mb}^2(\Omega;\mathbb{R}^2), \quad v(t_0)\in W_3^1(\Omega), \;\;\text{ and }\;\; \partial_t w(t_0)\in L_3(\Omega). \label{regt0}
\end{equation}
For further use, we also fix
\begin{equation}
\gamma\in (5/6,1), \label{exptheta}
\end{equation}
see Corollary~\ref{cor3.9} and Lemma~\ref{lem3.10}. For the remainder of this section, $C$ and $(C_i)_{i\ge 1}$ denote positive constants depending only on $\Omega$, $\theta$, $D$, $\alpha$, $\gamma$, $t_0$, $\|u(t_0)\|_{W_3^2(\Omega)}$, $\|v(t_0)\|_{W_3^1(\Omega)}$, and $\|w_0\|_{W_3^2(\Omega)}$. Dependence upon additional parameters is indicated explicitly. We first remark that the mass conservation \eqref{mass1}, in conjunction with the nonnegativity of $u$ and $v$, implies
\begin{equation}\label{mass2}
 \| u(t)\|_{L_1(\Omega)} \le M \qquad\mbox{and}\qquad \| v(t)\|_{L_1(\Omega)} \le M, \qquad t \in [t_0, T_m).
\end{equation}

We next derive a series of estimates. 

%%%%%%%%%%%%%%%%
\begin{lemma}\label{lem3.3}
  Let $T>t_0$. There is $C_1(T) >0$ such that, for $t \in [t_0,T] \cap [t_0,T_m)$,
  \begin{align*}
    \| L(u(t))\|_{L_1(\Omega)} + \| v(t)\|_{L_2(\Omega)} + \| w(t)\|_{W^1_2(\Omega)} + \| \partial_t w(t)\|_{L_2(\Omega)} & \le C_1(T),\\
    \int_{t_0}^t \left( \|\nabla \sqrt{u}(s)\|_{L_2(\Omega)}^2 + \| \partial_t w(s)\|_{W^1_2(\Omega)}^2 \right)\ \mathrm{d}s & \le  C_1(T),
  \end{align*}
  where $L$ is defined in \eqref{eL}.
\end{lemma}
%%%%%%%%%%%%%%%%

\begin{proof}
  Using \eqref{a1PL}, \eqref{a2PL}, \eqref{bPL} along with integration by parts, we obtain 
  \begin{align}\label{e3.3.1}
    & \hspace*{-20mm} \frac{\mathrm{d}}{\mathrm{d}t} \left( \| L(u)\|_{L_1(\Omega)} + \| L_\theta (v)\|_{L_1(\Omega)} \right) \nonumber \\
    &=  -\int_\Omega \frac{|\nabla u|^2}{u}\ \mathrm{d}x + \int_\Omega \nabla u \cdot \nabla w\ \mathrm{d}x + \int_\Omega (\ln u - \ln (\theta v))(\theta v -u)\ \mathrm{d}x \nonumber \\
    & =  -4 \|\nabla \sqrt{u}\|_{L_2(\Omega)}^2 - \int_\Omega u \Delta w\ \mathrm{d}x - \int_\Omega (\ln (\theta v)-\ln u)(\theta v -u)\ \mathrm{d}x \nonumber \\
    & \le  -4 \|\nabla \sqrt{u}\|_{L_2(\Omega)}^2 - \int_\Omega u \Delta w\ \mathrm{d}x\ ,
  \end{align}
  in view of the monotonicity of the logarithm. Due to \eqref{a2PL}--\eqref{bPL}, we further have
  \begin{align}\label{e3.3.2}
    \int_\Omega u \partial_t w\ \mathrm{d}x  = & \int_\Omega (\partial_t v \partial_t w + \theta v \partial_t w)\ \mathrm{d}x \nonumber \\
     = & \int_\Omega \partial_t w \left( \partial_t^2 w - D \Delta \partial_t w + \alpha \partial_t w \right)\ \mathrm{d}x
      + \theta \int_\Omega \partial_t w \left( \partial_t w - D \Delta w + \alpha w \right)\ \mathrm{d}x \nonumber \\
     = & \; \frac{1}{2} \frac{\mathrm{d}}{\mathrm{d}t} \left( \| \partial_t w \|_{L_2(\Omega)}^2 + \theta D \| \nabla w \|_{L_2(\Omega)}^2 + \theta \alpha \| w \|_{L_2(\Omega)}^2 \right)
    \nonumber \\  
    & + D \| \nabla \partial_t w \|_{L_2(\Omega)}^2 + (\alpha + \theta) \| \partial_t w \|_{L_2(\Omega)}^2.
  \end{align}
  Introducing
  \begin{equation*}
  Y := D \left(\| L(u)\|_{L_1(\Omega)} + \| L_\theta (v)\|_{L_1(\Omega)} \right) + \frac{1}{2} \left( \| \partial_t w \|_{L_2(\Omega)}^2 + \theta D \| \nabla w \|_{L_2(\Omega)}^2 + \theta \alpha \| w \|_{L_2(\Omega)}^2 \right)\ ,
  \end{equation*}
  we combine \eqref{e3.3.1} and \eqref{e3.3.2} and deduce from \eqref{a3PL} that
  \begin{align*}
  \frac{\mathrm{dY}}{\mathrm{d}t} + 4D \|\nabla\sqrt{u}\|_{L_2(\Omega)}^2 & + D \|\nabla\partial_t w\|_{L_2(\Omega)}^2 +(\alpha+\theta) \|\partial_t w\|_{L_2(\Omega)}^2 \\
  & \le \int_\Omega u\left( \partial_t w - D \Delta w \right)\ \mathrm{d}x = \int_\Omega u (v-\alpha w)\ \mathrm{d}x\ .
  \end{align*}
  Also, by \eqref{a2PL},
  \begin{equation*}
  \frac{1}{2} \frac{\mathrm{d}}{\mathrm{d}t} \| v \|_{L_2(\Omega)}^2 = \int_\Omega u v\ \mathrm{d}x - \theta \| v \|_{L_2(\Omega)}^2 \le \int_\Omega u v\ \mathrm{d}x\ .
  \end{equation*}
  Summing the previous two inequalities and using the nonnegativity of $u$ and $w$, we find
  \begin{align}
  \frac{\mathrm{d}}{\mathrm{d}t} \left( Y + \frac{1}{2} \| v \|_{L_2(\Omega)}^2 \right) 
  + 4D \|\nabla\sqrt{u}\|_{L_2(\Omega)}^2 & + D \|\nabla\partial_t w\|_{L_2(\Omega)}^2 +(\alpha+\theta) \|\partial_t w\|_{L_2(\Omega)}^2 \nonumber \\
  & \le 2 \int_\Omega uv\ \mathrm{d}x\ . \label{z1}
  \end{align}
  Since $\Omega\subset \mathbb{R}^2$, we infer from \eqref{mass2} and H\"{o}lder's, Gagliardo-Nirenberg's and Young's inequalities that
  \begin{align}
  2 \int_\Omega u v\ \mathrm{d}x &\le 2 \| u \|_{L_2(\Omega)} \| v \|_{L_2(\Omega)} = 2 \| v \|_{L_2(\Omega)} \| \sqrt{u} \|_{L_4(\Omega)}^2 \nonumber \\
  &\le C \| v \|_{L_2(\Omega)} \left( \| \nabla \sqrt{u} \|_{L_2(\Omega)} \| \sqrt{u} \|_{L_2(\Omega)} + \| \sqrt{u} \|_{L_2(\Omega)}^2 \right) \nonumber \\
  &\le C \| v \|_{L_2(\Omega)} \left( \sqrt{M} \| \nabla \sqrt{u} \|_{L_2(\Omega)} +M \right) \nonumber \\
  &\le 2D \| \nabla \sqrt{u} \|_{L_2(\Omega)}^2 + C \left( 1+  \| v \|_{L_2(\Omega)}^2 \right). \label{z2}
  \end{align}
Combining \eqref{z1} and \eqref{z2} and rearranging the terms, we conclude that
   \begin{align*}
    \frac{\mathrm{d}}{\mathrm{d}t} \left( Y + \frac{1}{2} \| v \|_{L_2(\Omega)}^2 \right) + 2D \|\nabla \sqrt{u}\|_{L_2(\Omega)}^2 & + D \| \nabla \partial_t w \|_{L_2(\Omega)}^2 + (\alpha + \theta) \| \partial_t w \|_{L_2(\Omega)}^2 \\
    &\le C \left( 1+  \| v \|_{L_2(\Omega)}^2 \right).
  \end{align*}
  Applying first Gronwall's inequality and then the time integrated version of the previous inequality, we deduce the claim in view of the positivity of all the terms involved in the left hand side of the above inequality and the finiteness of $Y(t_0)$ stemming from \eqref{regt0}.
\end{proof}

%%%%%%%%%%%%%%%%
\begin{corollary}\label{cor3.4}
  Let $T>t_0$.  There is $C_2(T) >0$ such that
  \begin{equation*}
  \| w(t)\|_{W^2_2(\Omega)} \le C_2(T) \, , \qquad t \in [t_0,T] \cap [t_0,T_m)\, .
  \end{equation*}
\end{corollary}
%%%%%%%%%%%%%%%%

\begin{proof}
  Using \eqref{a3PL} along with Lemma~\ref{lem3.3}, we obtain, for $t \in [t_0,T] \cap [t_0,T_m)$,
  \begin{equation*}
  D \| \Delta w(t)\|_{L_2(\Omega)} 
  \le \| \partial_t w(t)\|_{L_2(\Omega)} + \alpha \| w(t)\|_{L_2(\Omega)} 
  + \| v(t)\|_{L_2(\Omega)} \le (2+\alpha) C_1(T).
  \end{equation*}
 Since $w(t)\in W_{3,\mb}^2(\Omega)$ for $t\in [t_0,T]\cap [t_0,T_m)$ by Corollary~\ref{cor3.2}, the claim follows from Calderon-Zygmund's estimate which guarantees that $\| z\|_{W^2_2(\Omega)} \le C \left(\| \Delta z \|_{L_2(\Omega)} + \| z\|_{L_2(\Omega)} \right)$ for all $z \in W^2_{2,\mb}(\Omega)$.
\end{proof}

%%%%%%%%%%%%%%%%
\begin{lemma}\label{lem3.5}
  Let $T>t_0$ and $r>0$. There is $C_3(T,r) >0$ such that 
  $$\| u(t)\|_{L_{r+1}(\Omega)} + \| v(t)\|_{L_{r+1}(\Omega)} \le C_3(T,r) \, , \qquad
    t \in [t_0,T] \cap [t_0,T_m)\, .$$ 
\end{lemma}
%%%%%%%%%%%%%%%%

\begin{proof}
 {Let $r>0$}. Using \eqref{aPL} and integration by parts gives
 \begin{align*}
 \frac{\mathrm{d}}{\mathrm{d}t} \left( \| u\|_{L_{r+1}(\Omega)}^{r+1} + \theta^r \| v\|_{L_{r+1}(\Omega)}^{r+1} \right) & = -r(r+1) \int_\Omega u^{r-1} |\nabla u|^2\ \mathrm{d}x + r(r+1) \int_\Omega u^r \nabla u \cdot \nabla w\ \mathrm{d}x \\
 & \qquad + (r+1) \int_\Omega (u-\theta v) \left( \theta^r v^r - u^r \right)\ \mathrm{d}x \\
 & \le - \frac{4r}{r+1} \| \nabla (u^{(r+1)/2})\|_{L_2(\Omega)}^2 - r \int_\Omega u^{r+1} \Delta w \ \mathrm{d}x\ .
 \end{align*}
 It now follows from H\"{o}lder's and Gagliardo-Nirenberg's inequalities that 
\begin{align*}
r \left| \int_\Omega u^{r+1} \Delta w\ \mathrm{d}x \right| & \le r \| \Delta w\|_{L_2 (\Omega)} \left\| u^{(r+1)/2} \right\|_{L_4 (\Omega)}^2 \\
& \le r C \| \Delta w\|_{L_2 (\Omega)} \left(  \left\|\nabla u^{(r+1)/2} \right\|_{L_2 (\Omega)} \left\| u^{(r+1)/2} \right\|_{L_2 (\Omega)} + \left\| u^{(r+1)/2} \right\|_{L_2 (\Omega)}^2  \right) \\
& \le \frac{2r}{r+1} \left\|\nabla u^{(r+1)/2} \right\|_{L_2 (\Omega)} ^2 + C r(r+1) \|\Delta w\|_{L_2(\Omega)}^2 \left\| u^{(r+1)/2} \right\|_{L_2 (\Omega)}^2 \\
& \qquad + C r \| \Delta w\|_{L_2 (\Omega)} \left\| u^{(r+1)/2} \right\|_{L_2 (\Omega)}^2 \\
& \le \frac{2r}{r+1} \left\|\nabla u^{(r+1)/2} \right\|_{L_2 (\Omega)} ^2 + C r(r+1) \left( 1 + \|\Delta w\|_{L_2(\Omega)}^2 \right) \| u\|_{L_{r+1}(\Omega)}^{r+1}\ .
\end{align*} 
Combining the above two inequalities with Corollary~\ref{cor3.4}, we obtain 
 \begin{align*}
\frac{\mathrm{d}}{\mathrm{d}t} \left( \| u\|_{L_{r+1}(\Omega)}^{r+1} + \theta^r \| v\|_{L_{r+1}(\Omega)}^{r+1} \right) & \le - \frac{2r}{r+1} \left\|\nabla u^{(r+1)/2} \right\|_{L_2 (\Omega)} ^2 \\
& \qquad + C r(r+1) \left( 1 + C_2(T)^2 \right) \| u\|_{L_{r+1}(\Omega)}^{r+1}.
 \end{align*}
Then Gronwall's inequality implies the claim.
\end{proof}

%%%%%%%%%%%%%%%%
\begin{lemma}\label{lem3.6}
  Let $T>t_0$. There is $C_4(T) >0$ such that 
  $$\| u(t)\|_{L_\infty (\Omega)} + \| v(t)\|_{L_\infty (\Omega)} +  \| \partial_t v(t)\|_{L_\infty (\Omega)}\le C_4(T) \, , \qquad t \in [t_0,T] \cap [t_0,T_m)\, .$$ 
\end{lemma}
%%%%%%%%%%%%%%%%

\begin{proof}
  Let $T>t_0$ and $I := [t_0,T] \cap [t_0,T_m)$. In view of \eqref{a1PL} and \eqref{bPL}, $u$ satisfies 
  $$\partial_t u = \mathrm{div} \left( D \nabla u \right) + \mathrm{div}\, \mathbf{f} + g, \qquad (t,x) \in (t_0,T_m)\times \Omega,$$
  along with no-flux boundary conditions, where $D \equiv 1$, $\mathbf{f} := -u \nabla w$, and $g := \theta v -u$. According to Lemma~\ref{lem3.5},
  \begin{equation*}
  u \in L_\infty(I; L_1 (\Omega)) \;\;\text{ and }\;\; g \in L_\infty(I; L_3(\Omega))\ ,
  \end{equation*} 
  while Corollary~\ref{cor3.4}, Lemma~\ref{lem3.5}, and the continuous embedding of $W^2_2(\Omega)$ in $W^1_r (\Omega)$ for any $r \in [1, \infty)$ ensure that $\mathbf{f} \in  L_\infty(I; L_5(\Omega ; \RR^2))$. We then infer from \cite[Lemma~A.1]{TW2012} that
  $$\| u(t)\|_{L_\infty (\Omega)} \le C(T) \, , \qquad t \in [t_0,T] \cap [t_0,T_m)\, .$$ 
  Combining this estimate with \eqref{a2PL}, \eqref{regt0}, the nonnegativity of $v$, and the continuous embedding of $W_3^1(\Omega)$ in $L_\infty(\Omega)$, we further obtain
  \begin{align}
  0 \le v(t,x) & = e^{-\theta (t-t_0)} v(t_0,x) + \int_{t_0}^t u(s,x) e^{-\theta (t-s)}\ \mathrm{d}s \label{z3} \\
  & \le e^{-\theta (t-t_0)} \|v(t_0)\|_{L_\infty(\Omega)} + \int_{t_0}^t \|u(s)\|_{L_\infty(\Omega)} e^{-\theta (t-s)}\ \mathrm{d}s \nonumber \\
 & \le C \| v(t_0) \|_{W_3^1(\Omega)} + \frac{C(T) }{\theta} \le C(T) \nonumber 
 \end{align}
  for $(t,x) \in [t_0,T] \cap [t_0,T_m)$. Finally, the last part of the claim immediately follows from \eqref{a2PL}.
\end{proof}

%%%%%%%%%%%%%%%%
\begin{corollary}\label{cor3.7}
  Let $T>t_0$. There is $C_5(T) >0$ such that 
  \begin{equation*}
  \| w(t)\|_{W^1_\infty (\Omega)} \le C_5(T) , \qquad t \in [t_0,T] \cap [t_0,T_m).
  \end{equation*} 
\end{corollary}
%%%%%%%%%%%%%%%%

\begin{proof}
The proof is similar to that of \cite[Corollary~2.7]{La2019} and we recall it here for the sake of completeness. Owing to the continuous embedding of $W_3^{11/6}(\Omega)$ in $W_\infty^1(\Omega)$, see \cite[Theorem~7.1.2]{Pa1983}, and 
 \begin{equation*}
 \left( L_3(\Omega), W_{3,\mb}^2(\Omega) \right)_{11/12,3} \doteq W_{3,\mb}^{11/6}(\Omega),
 \end{equation*}
see \cite[Theorem~7.2]{Am1993}, it follows from \eqref{a3PL}, \eqref{regt0}, Duhamel's formula, the regularizing properties of the heat semigroup, see \cite[Theorem~V.2.1.3]{Am1995}, and Lemma~\ref{lem3.6} that, for $t\in [t_0,T]\cap [t_0,T_m)$, 
\begin{align*}
\|w(t)\|_{W_\infty^1(\Omega)} & \le C \| w(t)\|_{W_3^{11/6}(\Omega)} \\
& \le C e^{-\alpha (t-t_0)/2} \|w(t_0)\|_{W_3^{11/6}(\Omega)} + C \int_{t_0}^t e^{-\alpha(t-s)/2} (t-s)^{-11/12} \|v(s)\|_{L_3(\Omega)}\ \mathrm{d}s \\
& \le C \|w(t_0)\|_{W_3^2(\Omega)} + C C_4(T) \int_{t_0}^t  (t-s)^{-11/12}\ \mathrm{d}s \le C(T),
\end{align*}
and the proof is complete.
\end{proof}

%%%%%%%%%%%%%%%%
\begin{lemma}\label{lem3.8}
   Let $T>t_0$. There is $C_6(T) >0$ such that 
  $$\| \nabla u(t)\|_{L_2 (\Omega)} + \int_{t_0}^t \| \partial_t u(s)\|_{L_2 (\Omega)}^2\ \mathrm{d}s  \le C_6(T) \, , \qquad
    t \in [t_0,T] \cap [t_0,T_m)\, .$$ 
\end{lemma}
%%%%%%%%%%%%%%%%

\begin{proof}
  Multiplying \eqref{a1PL} by $\partial_t u$ and integrating over $\Omega$, we conclude from H\"{o}lder's and Young's inequalities that
  \begin{align*}
\| \partial_t u\|_{L_2 (\Omega)}^2  & + \frac{1}{2} \frac{\mathrm{d}}{\mathrm{d}t} \| u\|_{W_2^1 (\Omega)}^2 
    = - \int_\Omega \partial_t u (u \Delta w + \nabla u \cdot \nabla w)\ \mathrm{d}x + \theta \int_\Omega v \partial_t u\ \mathrm{d}x\\
    &\le \| \partial_t u\|_{L_2 (\Omega)} \| u\|_{L_\infty (\Omega)} \| \Delta w \|_{L_2 (\Omega)} 
    +  \| \partial_t u\|_{L_2 (\Omega)} \| \nabla u\|_{L_2 (\Omega)} \| \nabla w \|_{L_\infty (\Omega)} \\
    & \qquad + \theta  \| \partial_t u\|_{L_2 (\Omega)} \| v \|_{L_2 (\Omega)} \\
    &\le \frac{1}{2} \| \partial_t u\|_{L_2 (\Omega)}^2 + \frac{3}{2} \left( \| u\|_{L_\infty (\Omega)}^2 \| \Delta w \|_{L_2 (\Omega)}^2 
    +  \| \nabla u\|_{L_2 (\Omega)}^2 \| \nabla w \|_{L_\infty (\Omega)}^2 
    + \theta^2 \| v \|_{L_2 (\Omega)}^2  \right).
  \end{align*}
  Lemmas~\ref{lem3.3} and~\ref{lem3.6} and Corollaries~\ref{cor3.4} and~\ref{cor3.7} next imply that
  $$\| \partial_t u\|_{L_2 (\Omega)}^2  + \frac{\mathrm{d}}{\mathrm{d}t} \| u\|_{W_2^1(\Omega)}^2 \le C(T) \left( 1 +  \| \nabla u\|_{L_2 (\Omega)}^2 \right),$$
  so that Gronwall's inequality yields the claim.
\end{proof}

%%%%%%%%%%%%%%%%
\begin{corollary}\label{cor3.9}
  Let $T>t_0$. There is $C_7(T) >0$ such that 
  \begin{equation*}
  \| u(t)\|_{W^1_3 (\Omega)} + \|u(t)\|_{W^{2\gamma}_2(\Omega)} + \| v(t)\|_{W^1_3 (\Omega)} \le C_7(T), \qquad
    t \in [t_0,T] \cap [t_0,T_m),
    \end{equation*}
    the parameter $\gamma\in (5/6,1)$ being defined in \eqref{exptheta}.
\end{corollary}
%%%%%%%%%%%%%%%%

\begin{proof}
  Using \eqref{a1PL} along with properties of the heat semigroup (see, e.g., \cite[Proposition~12.5]{Am1983} and
  \cite[Theorem~V.2.1.3]{Am1995}) and H\"{o}lder's inequality, we deduce from \eqref{regt0}, Corollaries~\ref{cor3.4} and~\ref{cor3.7} as well as Lemmas~\ref{lem3.5}, \ref{lem3.6} and~\ref{lem3.8} that, for $t\in [t_0,T]\cap [t_0,T_m)$, 
  \begin{align*}
    \| \nabla u(t) \|_{L_3 (\Omega)} 
    &\le C \| \nabla u(t_0) \|_{L_3 (\Omega)} + C \int_{t_0}^t (t-s)^{-2/3} \| (u \Delta w + \nabla u\cdot \nabla w - \theta v + u)(s) \|_{L_2 (\Omega)}\ \mathrm{d}s \\
    &\le C + C \int_{t_0}^t (t-s)^{- 2/3} \left( \| u(s)\|_{L_\infty (\Omega)} \| \Delta w(s) \|_{L_2 (\Omega)} + \| u(s) \|_{L_2 (\Omega)} \right)\ \mathrm{d}s \\
    &  \qquad + C \int_{t_0}^t (t-s)^{- 2/3} \left( \| \nabla u(s)\|_{L_2 (\Omega)} \| \nabla w (s)\|_{L_\infty (\Omega)} + \theta \| v(s) \|_{L_2 (\Omega)} \right)\ \mathrm{d}s \\
    & \le C + C(T) \int_{t_0}^t (t-s)^{- 2/3}\ \mathrm{d}s \le C(T).
  \end{align*}
  Combining this estimate with \eqref{regt0} and \eqref{z3} , we obtain
  \begin{equation*}
  \| \nabla v(t) \|_{L_3(\Omega)} \le e^{-\theta (t-t_0)}\| \nabla v(t_0) \|_{L_3(\Omega)} + \int_{t_0}^t \| \nabla u(s) \|_{L_3(\Omega)} e^{-\theta (t-s)}\ \mathrm{d}s \le C(T).
  \end{equation*}
  Together with Lemma~\ref{lem3.5}, the above two estimates entail the stated $W_3^1$-bounds on $u$ and $v$.  
  
  We next invoke \eqref{a1PL} along with properties of the heat semigroup (see \cite[Theorem~V.2.1.3]{Am1995}) and H\"{o}lder's inequality to deduce 
  from Corollaries~\ref{cor3.4} and~\ref{cor3.7} and Lemmas~\ref{lem3.5}, \ref{lem3.6} and~\ref{lem3.8} that, for all $t \in [t_0,T] \cap [t_0,T_m)$,
  \begin{align*}
  \| u(t) \|_{W^{2\gamma}_2 (\Omega)} 
  &\le C \| u(t_0) \|_{W^{2\gamma}_2 (\Omega)} + C \int_{t_0}^t (t-s)^{-\gamma} 
  \| ( u \Delta w + \nabla u  \cdot \nabla w - \theta v + u)(s) \|_{L_2 (\Omega)} \ \mathrm{d}s \\
  &\le C + C \int_{t_0}^t (t-s)^{-\gamma} \| u(s)\|_{L_\infty (\Omega)} \| \Delta w(s) \|_{L_2 (\Omega)}\ \mathrm{d}s \\
  &  \quad + \int_{t_0}^t (t-s)^{-\gamma} \| \nabla u(s)\|_{L_2 (\Omega)} \| \nabla w (s)\|_{L_\infty (\Omega)} \ \mathrm{d}s \\
  & \quad + \int_{t_0}^t (t-s)^{-\gamma} \left( \theta \| v(s) \|_{L_2 (\Omega)} + \| u(s) \|_{L_2 (\Omega)} \right) \ \mathrm{d}s \\
  &\le C + C(T) \int_{t_0}^t (t-s)^{-\gamma} \ \mathrm{d}s \le C(T),
  \end{align*}
  and the proof is complete.
\end{proof}

In view of Corollaries~\ref{cor3.7} and~\ref{cor3.9} there is $C_8(T)>0$ such that
\begin{equation}\label{e3.9.1}
  \| u(t)\|_{W^1_3 (\Omega)} + \| v(t)\|_{W^1_3 (\Omega)} + \| w(t)\|_{W^1_3 (\Omega)} \le C_8(T) \, , \qquad
    t \in [t_0,T] \cap [t_0,T_m)\, .
\end{equation}
Hence, according to \eqref{e3.1.1}, we shall prove H\"{o}lder estimates with respect to time in order to conclude the global existence of $(u,v,w)$.

%%%%%%%%%%%%%%%%
\begin{lemma}\label{lem3.10}
  Let $T>t_0$ and $\delta := (3\gamma-2)/6\gamma$, recalling that $\gamma\in (5/6,1)$ is defined in \eqref{exptheta}. There is $C_9(T) >0$ such that 
  $$\| u(t_2)- u(t_1) \|_{W^1_3 (\Omega)} + \| v(t_2)- v(t_1)\|_{W^1_3 (\Omega)} + \| w(t_2)-w(t_1)\|_{W^1_3 (\Omega)} \le C_9(T) |t_2-t_1|^\delta$$
  for $t_1,t_2 \in [t_0,T] \cap [t_0,T_m)$.
\end{lemma}
%%%%%%%%%%%%%%%%

\begin{proof}
  The proof is similar to \cite[Lemma~2.10]{La2019}, but is recalled here for the sake of completeness. Let $T>0$ and $t_1,t_2 \in [t_0,T] \cap [t_0,T_m)$ such that $t_2>t_1$.
  
  First, by \eqref{a2PL} and \eqref{e3.9.1} we have
  \begin{equation*} 
  \| \partial_t v(t)\|_{W^1_3 (\Omega)} \le  \| u(t)\|_{W^1_3 (\Omega)} + \theta  \| v(t)\|_{W^1_3 (\Omega)} \le (1+\theta) C_8(T), \qquad t \in [t_0,T] \cap [t_0,T_m).
  \end{equation*}
  Hence,
  \begin{equation}\label{e3.10.1}
     \| v(t_2)- v(t_1)\|_{W^1_3 (\Omega)} \le \int_{t_1}^{t_2} \| \partial_t v(s)\|_{W^1_3 (\Omega)} \, \mathrm{d}s \le C(T) (t_2-t_1)\ .
  \end{equation}
  Furthermore, in view of H\"{o}lder's inequality, we obtain from Lemma~\ref{lem3.3} and Corollary~\ref{cor3.7}
  \begin{align}\label{e3.10.2}
    \| w(t_2)-w(t_1)\|_{W^1_3(\Omega)} 
    & \le C \| w(t_2)-w(t_1)\|_{W^1_\infty (\Omega)}^{1/3} \| w(t_2)-w(t_1)\|_{W^1_2 (\Omega)}^{2/3} \nonumber \\
    & \le C \left( \| w(t_1)\|_{W^1_\infty (\Omega)} + \| w(t_2)\|_{W^1_\infty (\Omega)} \right)^{1/3}
    \left( \int_{t_1}^{t_2} \| \partial_t w(s)\|_{W^1_2 (\Omega)} \, \mathrm{d}s \right)^{2/3} \nonumber \\
    &\le C (2 C_5(T))^{1/3}
    \left( \int_{t_1}^{t_2} \| \partial_t w(s)\|_{W^1_2 (\Omega)}^2 \, \mathrm{d}s \right)^{1/3} (t_2-t_1)^{1/3} \nonumber \\
    &\le C(T) (C_1(T))^{1/3} (t_2-t_1)^{1/3}.
  \end{align}
  Since $W^{4/3}_2(\Omega)$ is continuously embedded in $W^1_3(\Omega)$, interpolation inequalities (see \cite[Theorem~7.2]{Am1993}), H\"{o}lder's inequality, Lemma~\ref{lem3.8}, and Corollary~\ref{cor3.9} yield
  \begin{align}\label{e3.10.3}
    \| u(t_2)- u(t_1) \|_{W^1_3 (\Omega)} 
    &\le C \| u(t_2)- u(t_1) \|_{W^{4/3}_2(\Omega)} \nonumber \\
    &\le C \| u(t_2)- u(t_1) \|_{W^{2\gamma}_2 (\Omega)}^{2/3\gamma}
    \| u(t_2)- u(t_1) \|_{L_2 (\Omega)}^{2\delta} \nonumber \\
    &\le C \left( \| u(t_1) \|_{W^{2\gamma}_2 (\Omega)} +  \| u(t_2) \|_{W^{2\gamma}_2 (\Omega)} \right)^{2/3\gamma}  \left( \int_{t_1}^{t_2} \| \partial_t u(s) \|_{L_2 (\Omega)} \, \mathrm{d}s \right)^{2\delta} \nonumber \\
    &\le C (2C_7(T))^{2/3\gamma} \left( \int_{t_1}^{t_2} \| \partial_t u(s) \|_{L_2 (\Omega)}^2 \, \mathrm{d}s \right)^{\delta} (t_2-t_1)^{\delta} \nonumber \\
    &\le C(T) \left( C_6(T)\right)^{\delta} (t_2-t_1)^{\delta}\ .
  \end{align}
   Combining \eqref{e3.10.1}, \eqref{e3.10.2}, and \eqref{e3.10.3} completes the proof, since $0<\delta<1/3$.
\end{proof}

\begin{proof}[Proof of Theorem~\ref{theo1.1}]
  Proposition~\ref{prop3.1} and Corollary~\ref{cor3.2} imply the claim as soon as we know that $T_m = \infty$. Since $(u,v,w)\in BUC([0,t_0];W_3^1(\Omega;\mathbb{R}^3))$ by Proposition~\ref{prop3.1},  \eqref{e3.9.1} and Lemma~\ref{lem3.10} ensure that \eqref{e3.1.1} is satisfied for any $T>0$, and we indeed conclude that $T_m = \infty$.
\end{proof}

%%%%%%%%%%%%%%%%
%%%%%%%%%%%%%%%%
\section{Bounded solutions for small values of $M$}\label{sec4}
%%%%%%%%%%%%%%%%
%%%%%%%%%%%%%%%%

Let $M>0$. In this section, we assume that 
\begin{equation}\label{e4.0.1}
  (u_0,v_0,w_0) \in \mi_{M}, 
\end{equation}	
and denote the corresponding global solution to \eqref{aPL}--\eqref{iPL} by $(u,v,w)$, see Theorem~\ref{theo1.1}. Throughout this section, $b$ and $(b_i)_{i\ge 1}$ denote positive constants depending only on $\Omega$, $\theta$, $D$, $\alpha$, $u_0$, $v_0$, and $w_0$. Dependence upon additional parameters is indicated explicitly.

In order to show the boundedness of the solution to \eqref{aPL}--\eqref{iPL}, we first have a look at the evolution of the $L_1$-norms of $(u,v,w)$.

%%%%%%%%%%%%%%%%
\begin{lemma}\label{lem4.1}
For all $t \ge 0$,
	\begin{align}
	  &\| (u+v)(t) \|_{L_1(\Omega)} = M = \| u_0+v_0 \|_{L_1(\Omega)}, \label{e4.1.1} \\
		&\| u(t) \|_{L_1(\Omega)} \le M \qquad\mbox{and}\qquad \| v(t) \|_{L_1(\Omega)} \le M, \label{e4.1.2} \\
		& \| w(t) \|_{L_1(\Omega)} \le \| w_0 \|_{L_1(\Omega)} + \frac{M}{\alpha}. \label{e4.1.3} 
	\end{align}
\end{lemma}
%%%%%%%%%%%%%%%%

\begin{proof}
The identity \eqref{e4.1.1} is nothing but \eqref{mass}, and \eqref{e4.1.2} is a consequence thereof in view of the nonnegativity of $u$ and $v$. Using the nonnegativity of $v$ and $w$ as well as
\eqref{a3PL}, \eqref{bPL}, and \eqref{e4.1.2}, we obtain
$$
\frac{\mathrm{d}}{\mathrm{d}t} \| w \|_{L_1(\Omega)} + \alpha \| w \|_{L_1(\Omega)} = \| v \|_{L_1(\Omega)} \le M\ . $$
Hence
\begin{equation*}
\| w(t) \|_{L_1(\Omega)} \le e^{-\alpha t} \| w_0 \|_{L_1(\Omega)} + M \int_0^t e^{-\alpha(t-s)}\ \mathrm{d}s \le \| w_0 \|_{L_1(\Omega)} + \frac{M}{\alpha}
\end{equation*}
for all $t \ge 0$ as $\alpha >0$. 
\end{proof}

Next, as in \cite{BN1993, GZ1998, NSY1997}, we use the structure of the Liapunov functional $\ml$ defined in \eqref{dPL} and follow the strategy from \cite{La2019} (see also \cite{CS2015}). We begin with a lower bound on $\ml$ for $M$ appropriately small which relies on the Trudinger-Moser inequality and first note the following consequence thereof (see \cite[Proposition~2.3]{CY1988} and \cite[Section~2]{NSY1997}).

%%%%%%%%%%%%%%%%
\begin{proposition}\label{prop4.3}
  There is $K_0 >0$ depending only on $\Omega$ such that
	$$\int_\Omega e^{|z|} \ \mathrm{d}x \le K_0 \, \exp \left(\frac{\| \nabla z\|_{L_2(\Omega)}^2}{8\pi} + \frac{\| z\|_{L_1(\Omega)}}{|\Omega|} \right) \qquad\mbox{for all } z \in W^1_2 (\Omega).$$	
\end{proposition}
%%%%%%%%%%%%%%%%

%%%%%%%%%%%%%%%%
\begin{lemma}\label{lem4.4}
  Assume that \eqref{e4.0.1} is satisfied. There is $b_1>0$ such that, for all $t \ge 0$,
	\begin{align}
	  \ml (u(t), v(t), w(t)) & \le \ml (u_0,v_0,w_0) <\infty, \label{e4.4.1} \\
		\ml (u(t),v(t),w(t)) & \ge \frac{4\pi(1+\theta)D -M}{8\pi} \|\nabla w(t)\|_{L_2 (\Omega)}^2 + \frac{\alpha(1+\theta)}{2} \|w(t)\|_{L_2 (\Omega)}^2 \nonumber \\
		& \quad + \frac{1}{2} \|\left(D\Delta w - \alpha w + v \right)(t)\|_{L_2 (\Omega)}^2 - b_1. \label{e4.4.2}
	\end{align}
\end{lemma} 
%%%%%%%%%%%%%%%%

\begin{proof}
We first observe that \eqref{e4.0.1} guarantees that $\ml (u_0,v_0,w_0) \in \mathbb{R}$, while Theorem~\ref{theo1.1} ensures that $\ml (u(t), v(t), w(t)) \in \mathbb{R}$ for all $t>0$. 
 
 Next, using the convexity of $s\mapsto -\ln (s)$ and Jensen's inequality as well as the nonnegativity of $u$ and $v$, we obtain
	\begin{align}\label{e4.4.3}
	  0 &=  - \|u\|_{L_1(\Omega)} \ln \left( \int_\Omega \frac{u}{\|u\|_{L_1(\Omega)}} \frac{e^w \|u\|_{L_1(\Omega)}}{u \|e^w\|_{L_1(\Omega)}}\ \mathrm{d}x  \right) \nonumber \\
		&\le - \|u\|_{L_1(\Omega)} \int_\Omega \frac{u}{\|u\|_{L_1(\Omega)}} \, \ln \left( \frac{e^w \|u\|_{L_1(\Omega)}}{u \|e^w\|_{L_1(\Omega)}} \right)\ \mathrm{d}x  \nonumber \\
		&= \int_\Omega (u \ln u - uw)\ \mathrm{d}x + \|u\|_{L_1(\Omega)} \ln \left( \|e^w\|_{L_1(\Omega)} \right) - \|u\|_{L_1(\Omega)} \ln \left( \|u\|_{L_1(\Omega)} \right),
	\end{align}
	as well as
	\begin{align}\label{e4.4.4}
	  0 &=  - \|v\|_{L_1(\Omega)} \ln \left( \int_\Omega \frac{v}{\|v\|_{L_1(\Omega)}} \frac{e^w \|\theta v\|_{L_1(\Omega)}}{\theta v \|e^w\|_{L_1(\Omega)}}\ \mathrm{d}x \right) \nonumber \\
		&\le - \|v\|_{L_1(\Omega)} \int_\Omega \frac{v}{\|v\|_{L_1(\Omega)}} \, \ln \left( \frac{e^w \|\theta v\|_{L_1(\Omega)}}{\theta v \|e^w\|_{L_1(\Omega)}} \right)\ \mathrm{d}x  \nonumber \\
		&= \int_\Omega (v \ln (\theta v) - vw)\ \mathrm{d}x + \|v\|_{L_1(\Omega)} \ln \left( \|e^w\|_{L_1(\Omega)} \right) - \|v\|_{L_1(\Omega)} \ln \left( \|\theta v\|_{L_1(\Omega)} \right).
	\end{align}
	In view of the nonnegativity of $(u,v,w)$, $L$, and $L_\theta$, we infer from \eqref{dPL}, \eqref{e4.1.1}, \eqref{e4.4.3}, \eqref{e4.4.4}, and Proposition~\ref{prop4.3} that
	\begin{align}\label{e4.4.5}
	  \ml (u,v,w) & = \int_\Omega (u \ln{u} - u w )\ \mathrm{d}x + \int_\Omega (v\ln{(\theta v)} - vw)\ \mathrm{d}x - \|u\|_{L_1(\Omega)} + |\Omega| - \|v\|_{L_1(\Omega)} + \frac{|\Omega|}{\theta} \nonumber \\
	  & \quad + \frac{1+\theta}{2} \left( D \|\nabla w\|_{L_2 (\Omega)}^2 + \alpha \|w\|_{L_2 (\Omega)}^2 \right) + \frac{1}{2} \|D\Delta w - \alpha w + v\|_{L_2 (\Omega)}^2 \nonumber \\
	  & \ge -M  \ln \left( \|e^w\|_{L_1(\Omega)} \right) + L\left( \|u\|_{L_1(\Omega)} \right) 
		+ L_\theta\left( \|v\|_{L_1(\Omega)} \right) + (|\Omega|-1)\left( 1+ \frac{1}{\theta} \right) \nonumber \\
		& \quad + \frac{1+\theta}{2} \left( D \|\nabla w\|_{L_2 (\Omega)}^2 + \alpha \|w\|_{L_2 (\Omega)}^2 \right) + \frac{1}{2} \|D\Delta w - \alpha w + v\|_{L_2 (\Omega)}^2 \nonumber \\
		& \ge \frac{4\pi(1+\theta)D -M}{8\pi} \|\nabla w\|_{L_2 (\Omega)}^2 + \frac{\alpha(1+\theta)}{2} \|w\|_{L_2 (\Omega)}^2 + \frac{1}{2} \|D\Delta w - \alpha w + v\|_{L_2 (\Omega)}^2 \nonumber \\
		& \quad - M \ln K_0 - \frac{M}{|\Omega|} \|w\|_{L_1 (\Omega)} + (|\Omega|-1)\left( 1+ \frac{1}{\theta} \right).
	\end{align}
	Inserting \eqref{e4.1.3}, we obtain \eqref{e4.4.2}, while \eqref{e4.4.1} immediately follows from \eqref{e4.2.1} since $\md \ge 0$.
\end{proof}

Hence, for $M \in (0, 4\pi(1+\theta)D)$, the Liapunov functional is bounded from below and we obtain further refined estimates.

%%%%%%%%%%%%%%%%
\begin{lemma}\label{lem4.5}
  Assume that \eqref{e4.0.1} is satisfied with $M \in (0, 4\pi(1+\theta)D)$. There is $b_2>0$ such that, for all $t \ge 0$,
	\begin{align*}
	  \|u(t) \ln (u(t)) \|_{L_1 (\Omega)} + \|v(t) \ln (v(t)) \|_{L_1 (\Omega)} + \|w(t) \|_{W^1_2 (\Omega)} + \| \partial_t w(t) \|_{L_2 (\Omega)} & \le b_2,\\
	  \int_0^\infty \|\partial_t w(s) \|_{W^1_2 (\Omega)}^2 \ \mathrm{d}s &\le b_2.
	\end{align*}
\end{lemma}
%%%%%%%%%%%%%%%%

\begin{proof}
  Since \eqref{a3PL} and Lemma~\ref{lem4.4} imply that
	$$\min \left\{ \frac{4\pi(1+\theta)D -M}{8\pi}, \frac{\alpha(1+\theta)}{2}, \frac{1}{2} \right\} \left( \|w(t) \|_{W^1_2 (\Omega)}^2 + \| \partial_t w(t) \|_{L_2 (\Omega)}^2\right)
	\le b_1 + \ml (u_0,v_0,w_0)$$
	and $M \in (0, 4\pi(1+\theta)D)$, we get
	\begin{equation}\label{e4.5.1}
	  \|w(t) \|_{W^1_2 (\Omega)} + \| \partial_t w(t) \|_{L_2 (\Omega)} \le b, \qquad t \ge 0. 
	\end{equation}	
  Next, \eqref{dPL}, \eqref{e4.1.1}, \eqref{e4.1.2}, \eqref{e4.4.1}, and the Young inequality $AB \le e^{A-1}+ B \ln B$ for $A,B >0$ yield
	\begin{align*}
	  & \int_\Omega \left( u \ln u + v \ln (\theta v) \right)\ \mathrm{d}x \\
		&\quad \le \ml (u, v, w) + \int_\Omega (u+v)\ \mathrm{d}x + \int_\Omega (u+v)w\ \mathrm{d}x \\
		&\quad \le \ml (u_0,v_0,w_0) + M + \int_\Omega \left( \frac{u}{2} \ln \left(\frac{u}{2} \right)  + \frac{v}{2} \ln \left(\frac{v}{2} \right) + 2 e^{2w-1}\right)\ \mathrm{d}x \\
		&\quad \le \ml (u_0,v_0,w_0) + M + \frac{1}{2} \int_\Omega \left( u \ln u  + v \ln (\theta v) \right)\ \mathrm{d}x + 2 \int_\Omega e^{2w-1}\ \mathrm{d}x - \frac{\ln\theta}{2} \int_\Omega v\ \mathrm{d}x \\
		&\quad \le \ml (u_0,v_0,w_0) + \frac{1}{2} \int_\Omega \left( u \ln u  + v \ln (\theta v) \right)\ \mathrm{d}x + 2 \int_\Omega e^{2w}\ \mathrm{d}x + M (1 + |\ln\theta|)\ . 
	\end{align*}
	In view of Proposition~\ref{prop4.3} this implies
	$$\int_\Omega \left( u \ln u + v \ln (\theta v) \right)\ \mathrm{d}x \le 4 K_0 \exp \left(\frac{\| \nabla w\|_{L_2(\Omega)}^2}{2\pi} + 2\frac{\| w\|_{L_1(\Omega)}}{|\Omega|} \right) + b\ ,$$
	so that, in view of $z \ln z \ge -1/e$ for all $z \ge 0$, we conclude from \eqref{e4.1.2}, \eqref{e4.1.3}, and \eqref{e4.5.1} that
	\begin{align}\label{e4.5.2}
	   \|u \ln u \|_{L_1 (\Omega)} & + \|v \ln v \|_{L_1 (\Omega)} \nonumber \\
	  & \le \int_\Omega \left( u \ln u + v \ln v \right)\ \mathrm{d}x +  \frac{4 |\Omega|}{e} \nonumber \\
	  & \le \int_\Omega \left( u \ln u + v \ln (\theta v) \right)\ \mathrm{d}x - \ln \theta \int_\Omega v\ \mathrm{d}x + \frac{4 |\Omega|}{e} \nonumber \\
	  & \le \int_\Omega \left( u \ln u + v \ln (\theta v) \right)\ \mathrm{d}x +M | \ln \theta | + \frac{4 |\Omega|}{e}\le b\ , \qquad t \ge 0\ .
	\end{align}
	Finally, we deduce from \eqref{a3PL}, \eqref{ePL}, and Lemmas~\ref{lem4.2} and \ref{lem4.4} that
	\begin{align*}
	  \min \{D, 1+\theta+ \alpha\} \int_0^t \|\partial_t w(s) \|_{W^1_2 (\Omega)}^2 \ \mathrm{d}s 
		& \le \int_0^t \md (u(s), v(s), w(s)) \ \mathrm{d}s \\
		& \le \ml (u_0,v_0,w_0) - \ml (u(t),v(t),w(t)) \\
		& \le \ml (u_0,v_0,w_0) +b_1 
	\end{align*}
	which, in view of \eqref{e4.5.1} and \eqref{e4.5.2}, completes the proof.
\end{proof}

Starting from the previous estimates, we derive further time-independent estimates for the solution by using the ideas from \cite{La2019} as well as the following inequality, which is
\cite[Equation~(22)]{BHN1994}:

Given $\eta >0$, there is a positive constant $\kappa_\eta$ depending only on $\eta$ and $\Omega$ such that
\begin{equation}\label{e4.5.3}
  \| z \|_{L_3 (\Omega)}^3 \le \eta \| z \|_{W^1_2 (\Omega)}^2 \| z \ln |z| \|_{L_1 (\Omega)} + \kappa_\eta \| z \|_{L_1 (\Omega)} \qquad\mbox{for all } z \in W^1_2(\Omega).
\end{equation}

%%%%%%%%%%%%%%%%
\begin{lemma}\label{lem4.6}
  Assume that \eqref{e4.0.1} is satisfied with $M \in (0, 4\pi(1+\theta)D)$. There is $b_3>0$ such that, for all $t \ge 0$,
	$$\| u(t) \|_{L_2 (\Omega)} + \| v(t) \|_{L_3 (\Omega)} \le b_3\ .$$
\end{lemma}
%%%%%%%%%%%%%%%%

\begin{proof}
It follows from \eqref{aPL}, \eqref{e4.1.2}, H\"{o}lder's and Young's inequalities, and the nonnegativity of $u$ and $w$ 
\begin{align*}
\frac{\mathrm{d}}{\mathrm{d}t} \| u \|_{L_2 (\Omega)}^2 & + 2 \| \nabla u \|_{L_2 (\Omega)}^2 = \int_\Omega \left( - u^2 \Delta w + 2\theta uv - 2u^2\right)\ \mathrm{d}x \\
&\le  \frac{1}{D} \int_\Omega u^2\left(v-\alpha w - \partial_t w \right)\ \mathrm{d}x - 2 \left\| u - \frac{\theta v}{2} \right\|_{L_2(\Omega)}^2 + \frac{\theta^2}{2} \| v\|_{L_2(\Omega)}^2 \\
&\le \frac{1}{D} \| v \|_{L_3 (\Omega)} \| u \|_{L_3 (\Omega)}^2 + \frac{1}{D} \| \partial_t w\|_{L_2 (\Omega)} \| u \|_{L_4 (\Omega)}^2 + \frac{2\theta}{3} \| v \|_{L_3 (\Omega)}^3 + \frac{\theta^4}{3} \\
&\le \theta \| v \|_{L_3 (\Omega)}^3 + \frac{2}{3\theta^{1/2} D^{3/2}} \| u \|_{L_3 (\Omega)}^3 + \frac{1}{D} \| \partial_t w\|_{L_2 (\Omega)} \| u \|_{L_4 (\Omega)}^2 + b \ .
\end{align*}
Next, Gagliardo-Nirenberg's inequality yields
\begin{align*}
\frac{\mathrm{d}}{\mathrm{d}t} \| u \|_{L_2 (\Omega)}^2 + 2 \| \nabla u \|_{L_2 (\Omega)}^2 & \le \theta \|v\|_{L_3(\Omega)}^3 + b \left( 1 + \|u\|_{L_3(\Omega)}^3 \right) \\
& \quad + b \| \partial_t w\|_{L_2 (\Omega)} \left( \| \nabla u \|_{L_2 (\Omega)} \| u \|_{L_2 (\Omega)} + \| u \|_{L_2 (\Omega)}^2 \right) \\
& \le \theta \| v \|_{L_3 (\Omega)}^3 + b \left( 1 + \| u \|_{L_3 (\Omega)}^3 \right) + \frac{1}{2} \| \nabla u \|_{L_2 (\Omega)}^2 \\
& \quad + b \left(\| \partial_t w\|_{L_2 (\Omega)} + \| \partial_t w\|_{L_2 (\Omega)}^2\right) \| u \|_{L_2 (\Omega)}^2\ .
\end{align*}
In view of Lemma~\ref{lem4.5}, this implies
\begin{equation}\label{e4.6.1}
\frac{\mathrm{d}}{\mathrm{d}t} \| u \|_{L_2 (\Omega)}^2 + \frac{3}{2} \| \nabla u \|_{L_2 (\Omega)}^2 \le \theta \| v \|_{L_3 (\Omega)}^3 + b \left( 1 + \| u \|_{L_3 (\Omega)}^3 + b_2(1+b_2) \| u \|_{L_2 (\Omega)}^2 \right)  \ .
\end{equation}
Moreover, using \eqref{a2PL} along with Young's inequality, we obtain
\begin{equation*}
\frac{1}{3} \frac{\mathrm{d}}{\mathrm{d}t} \| v \|_{L_3 (\Omega)}^3 + \theta \| v \|_{L_3 (\Omega)}^3 = \int_\Omega uv^2\ \mathrm{d}x \le \frac{4}{3\theta^2} \| u \|_{L_3 (\Omega)}^3 +  \frac{\theta}{3} \| v \|_{L_3 (\Omega)}^3\ .
\end{equation*}
Hence,
\begin{equation}\label{e4.6.2}
\frac{\mathrm{d}}{\mathrm{d}t} \| v \|_{L_3 (\Omega)}^3 + 2\theta \| v \|_{L_3 (\Omega)}^3 \le \frac{4}{\theta^2} \| u \|_{L_3 (\Omega)}^3\ .
\end{equation}
Defining $Y := \| u \|_{L_2 (\Omega)}^2 + \| v \|_{L_3 (\Omega)}^3$, we deduce from \eqref{e4.6.1}, \eqref{e4.6.2} along with \eqref{e4.1.2}, \eqref{e4.5.3}, and Lemma~\ref{lem4.5} that, for $\eta >0$,
\begin{align*}
\frac{\mathrm{d}Y}{\mathrm{d}t} + \frac{3}{2} \| \nabla u \|_{L_2 (\Omega)}^2 + \theta \| v \|_{L_3 (\Omega)}^3 &\le b \left( 1 + \| u \|_{L_3 (\Omega)}^3 + \| u \|_{L_2 (\Omega)}^2 \right) \\
&\le b \left( 1 + \eta b_2 \| u \|_{W^1_2 (\Omega)}^2 + \kappa_\eta M + \| u \|_{L_2 (\Omega)}^2 \right) \\
&\le \eta b_4 \| \nabla u \|_{L_2 (\Omega)}^2 + b \left( 1 + \| u \|_{L_2 (\Omega)}^2 +\kappa_\eta \right)\ .
\end{align*}
Choosing $\eta := 1/2b_4>0$, we find
\begin{align*}
 \frac{\mathrm{d}Y}{\mathrm{d}t} + \| \nabla u \|_{L_2 (\Omega)}^2 + \theta \| v \|_{L_3 (\Omega)}^3
\le b \left( 1 + \| u \|_{L_2 (\Omega)}^2 \right)\ .
\end{align*}
Owing to the Gagliardo-Nirenberg and Young inequalities along with \eqref{e4.1.2}, 
\begin{align*}
\|u\|_{L_2(\Omega)}^4 & \le b \left( \|\nabla u\|_{L_2(\Omega)}^2 M^2 + M^4 \right)\ , \\
\|u\|_{L_2(\Omega)}^2 & \le b_5 \|u\|_{L_2(\Omega)}^4 + \frac{1}{b_5}\ ,
\end{align*}
so that, using once more Young's inequality, 
\begin{align*}
\frac{\mathrm{d}Y}{\mathrm{d}t} + \|u\|_{L_2(\Omega)}^2 + b_5 \|u\|_{L_2(\Omega)}^4 + \theta \| v \|_{L_3 (\Omega)}^3 & \le \frac{\mathrm{d}Y}{\mathrm{d}t} + 2b_5 \|u\|_{L_2(\Omega)}^4 + \theta \| v \|_{L_3 (\Omega)}^3 + \frac{1}{b_5} \\
& \le b \left( 1 + \| u \|_{L_2 (\Omega)}^2 \right) + \frac{1}{b_5} \le b_5 \|u\|_{L_2(\Omega)}^4 + b \ .
\end{align*}
Consequently,
\begin{align*}
 \frac{\mathrm{d}Y}{\mathrm{d}t} + \min\{1,\theta\} Y \le b\ , \qquad t>0\ ,
\end{align*}
from which Lemma~\ref{lem4.6} follows after integration with respect to time.
\end{proof}

%%%%%%%%%%%%%%%%
\begin{corollary}\label{cor4.7}
 Assume that \eqref{e4.0.1} is satisfied with $M \in (0, 4\pi(1+\theta)D)$. There is $b_6 >0$ such that, for all $t \ge 0$,
\begin{equation*}
\|\Delta w(t)\|_{L_2(\Omega)} + \| w(t) \|_{W^1_3 (\Omega)}  \le b_6.
\end{equation*}
\end{corollary}
%%%%%%%%%%%%%%%%
\begin{proof}
First, by \eqref{a3PL}, Lemma~\ref{lem4.5}, and Lemma~\ref{lem4.6}, 
\begin{align*}
D \|\Delta w\|_{L_2(\Omega)} & = \|\partial_t w + \alpha w - v\|_{L_2(\Omega)} \le \|\partial_t w\|_{L_2(\Omega)} + \|\alpha w\|_{L_2(\Omega)} + \|v\|_{L_2(\Omega)} \\
& \le (1+\alpha) b_2 + |\Omega|^{1/6} b_3.
\end{align*}
	
Next, let $t>0$. We infer from \eqref{a3PL}, Duhamel's formula, the regularizing effect of the heat semigroup, and Lemma~\ref{lem4.6} that
\begin{align*}
\| w(t)\|_{W^1_3 (\Omega)} & \le b e^{-\alpha t/2} \| w_0\|_{W^1_{3} (\Omega)} + b \int_0^t e^{-\alpha(t-s)/2} (t-s)^{-1/2} \| v(s)\|_{L_{3} (\Omega)}\ \mathrm{d}s \\
& \le b \left( 1 + b_3 \int_0^\infty e^{-\alpha s/2} s^{-1/2}\ \mathrm{d}s \right),
\end{align*}
which completes the proof.
\end{proof}

The previous result allows us to obtain better $L_p$-estimates for $u$.

%%%%%%%%%%%%%%%%
\begin{lemma}\label{lem4.8}
Assume that \eqref{e4.0.1} is satisfied with $M \in (0, 4\pi(1+\theta)D)$. For any $p \in [1,\infty)$, there is $b_7(p)>0$  such that, for all $t \ge 0$,
	$$\| u(t) \|_{L_p (\Omega)} \le b_7(p)\ .$$
\end{lemma}
%%%%%%%%%%%%%%%%

\begin{proof}
We fix $p > 2$. Using \eqref{a1PL}, \eqref{bPL}, H\"{o}lder's inequality, and Corollary~\ref{cor4.7}, we have
\begin{align}\label{e4.8.1}
\frac{1}{p} \frac{\mathrm{d}}{\mathrm{d}t} \| u \|_{L_p (\Omega)}^p &= -(p-1) \int_\Omega u^{p-2} |\nabla u|^2\ \mathrm{d}x + (p-1) \int_\Omega u^{p-1} \nabla u \cdot \nabla w\ \mathrm{d}x + \int_\Omega u^{p-1} (\theta v - u)\ \mathrm{d}x \nonumber \\
& = -\frac{4(p-1)}{p^2} \| \nabla (u^{p/2}) \|_{L_2 (\Omega)}^2 - \frac{p-1}{p} \int_\Omega u^p \Delta w\ \mathrm{d}x + \int_\Omega u^{p-1} (\theta v - u)\ \mathrm{d}x \nonumber \\
& \le -\frac{4(p-1)}{p^2} \| \nabla (u^{p/2}) \|_{L_2 (\Omega)}^2 + \frac{p-1}{p} b_6 \| u^{p/2} \|_{L^4(\Omega)}^2 + \int_\Omega u^{p-1} (\theta v - u)\ \mathrm{d}x .
\end{align}
Also, it follows from \eqref{a2PL} and Young's inequality that
\begin{align}\label{z4}
\frac{1+\theta^{p-1}}{p} \frac{\mathrm{d}}{\mathrm{d}t} \| v \|_{L_p (\Omega)}^p & = - \int_\Omega (\theta v)^{p-1} (\theta v - u)\ \mathrm{d}x - \theta \|v\|_{L_p(\Omega)}^p + \int_\Omega u v^{p-1}\ \mathrm{d}x \nonumber \\
& \le - \int_\Omega (\theta v)^{p-1} (\theta v - u)\ \mathrm{d}x - \frac{\theta}{p} \|v\|_{L_p(\Omega)}^p + \frac{1}{p \theta^{p-1}} \|u\|_{L_p(\Omega)}^p.
\end{align}
Next, in view of the Gagliardo-Nirenberg inequality (see, e.g., \cite[Lemma~2.3]{LL2016} for a version involving $L_q$-spaces for $q >0$), $p>2$, Lemma~\ref{lem4.6}, and Young's inequality, we further obtain 
\begin{align}\label{e4.8.2}
\frac{p-1}{p} b_6 \| u^{p/2} \|_{L_4 (\Omega)}^2  & \le b(p) \left( \| \nabla (u^{p/2}) \|_{L_2 (\Omega)}^{(p-1)/p} \| u^{p/2} \|_{L_{4/p} (\Omega)}^{1/p} + \| u^{p/2} \|_{L_{4/p} (\Omega)}  \right)^2 \nonumber \\
&\le b(p) \left( \| \nabla (u^{p/2}) \|_{L_2 (\Omega)}^{2(p-1)/p} \|u\|_{L_2(\Omega)} + \|u\|_{L_2(\Omega)}^p  \right) \nonumber \\
&\le b(p) \left( b_3 \| \nabla (u^{p/2}) \|_{L_2 (\Omega)}^{2(p-1)/p} + b_3^p \right)  \nonumber \\
&\le \frac{2(p-1)}{p^2} \| \nabla (u^{p/2}) \|_{L_2 (\Omega)}^2 + b(p).
\end{align}
Combining \eqref{e4.8.1}, \eqref{z4}, and \eqref{e4.8.2} and using the monotonicity of $z\mapsto z^{p-1}$ on $[0,\infty)$, we deduce that
\begin{equation}\label{z5}
\frac{1}{p} \frac{\mathrm{d}y}{\mathrm{d}t} + \frac{2(p-1)}{p^2} \| \nabla (u^{p/2}) \|_{L_2(\Omega)}^2 + \frac{\theta}{p} \|v\|_{L_p(\Omega)}^p \le b_8(p) \left( 1 + \| u \|_{L_p (\Omega)}^p \right)\ , 
\end{equation}
with $y:= \| u \|_{L_p (\Omega)}^p + (1+\theta^{p-1}) \| v \|_{L_p (\Omega)}^p$. Using once more Lemma~\ref{lem4.6} and Gagliardo-Nirenberg's and Young's inequalities gives
\begin{align*}
\| u\|_{L^p(\Omega)}^p & = \| u^{p/2} \|_{L_2 (\Omega)}^2 \le b(p) \left( \| \nabla (u^{p/2}) \|_{L_2 (\Omega)}^{(p-2)/p} \| u^{p/2} \|_{L_{4/p} (\Omega)}^{2/p} + \| u^{p/2} \|_{L_{4/p} (\Omega)}  \right)^2 \\
& \le b(p) \left( \| \nabla (u^{p/2}) \|_{L_2 (\Omega)}^{2(p-2)/p} \| u^{p/2} \|_{L_{4/p} (\Omega)}^{4/p} + \| u^{p/2} \|_{L_{4/p} (\Omega)}^2  \right)\\
& \le b(p) \left( b_3^2 \| \nabla (u^{p/2}) \|_{L_2 (\Omega)}^{2(p-2)/p} + b_3^p \right) \\
& \le \frac{(p-1)}{p^2 b_8(p)} \| \nabla (u^{p/2}) \|_{L_2 (\Omega)}^2 + b(p)\ ,
\end{align*}
and we infer from \eqref{z5} that
\begin{align*}
& \frac{1}{p} \frac{\mathrm{d}y}{\mathrm{d}t} + \frac{p-1}{p^2} \| \nabla (u^{p/2}) \|_{L_2(\Omega)}^2 + b_8(p) \| u\|_{L^p(\Omega)}^p + \frac{\theta}{p} \|v\|_{L_p(\Omega)}^p \\
& \le \frac{1}{p} \frac{\mathrm{d}y}{\mathrm{d}t} + \frac{2(p-1)}{p^2} \| \nabla (u^{p/2}) \|_{L_2(\Omega)}^2 + \frac{\theta}{p} \|v\|_{L_p(\Omega)}^p + b(p) \\
& \le b_8(p) \left( 1 + \| u \|_{L_p (\Omega)}^p \right) + b(p) \le \frac{(p-1)}{p^2} \| \nabla (u^{p/2}) \|_{L_2 (\Omega)}^2 + b(p)\ .
\end{align*}
Hence, 
\begin{equation*}
\frac{\mathrm{d}y}{\mathrm{d}t} + \min\left\{ p b_8(p) , \frac{\theta}{1+\theta^{p-1}} \right\} y \le b(p)\ , \qquad t\ge 0\ .
\end{equation*}
Integration with respect to time of the above differential inequality completes the proof, after noticing that $u_0\in L_p (\Omega)$ due to the continuous embedding of $W^1_3 (\Omega)$ in $L_\infty(\Omega)$.
\end{proof}

We are now in a position to prove Theorem~\ref{theo1.2}.

\begin{proof}[Proof of Theorem~\ref{theo1.2}]
%\begin{enumerate} \item[(a)] 

\noindent (a) We fix $p\in (2,3)$ and $\xi\in (2/p,1)$. Since
\begin{equation*}
\left( W_{p,\mb}^{-1}(\Omega) , W_{p,\mb}^1(\Omega) \right)_{(1+\xi)/2,p} \doteq W_{p,\mb}^\xi(\Omega)
\end{equation*}
by \cite[Theorem~7.2]{Am1993} and writing \eqref{a1PL} as
\begin{equation*}
\partial_t u - \Delta u + u = \theta v - \mathrm{div}(u\nabla w) \qquad\;\; \text{ in }\;\; (0,\infty)\times \Omega,
\end{equation*} 
we infer from Duhamel's formula and \cite[Theorem~V.2.1.3]{Am1995} that, for $t\ge 0$, 
\begin{align*}
\|u(t)\|_{W_p^\xi(\Omega)} & \le b e^{-t/2} \|u_0\|_{W_p^\xi(\Omega)} + b \int_0^t (t-s)^{-(1+\xi)/2} e^{-(t-s)/2} \| \theta v(s) - \mathrm{div}(u\nabla w)(s) \|_{W_{p,\mb}^{-1}(\Omega)}\ \mathrm{d}s \\ 
& \le b \|u_0\|_{W_3^1(\Omega)} + b \int_0^t (t-s)^{-(1+\xi)/2} e^{-(t-s)/2} \left( \|v(s)\|_{L_p(\Omega)} + \|(u\nabla w)(s) \|_{L_p(\Omega)} \right)\ \mathrm{d}s.
\end{align*}
We next deduce from Lemma~\ref{lem4.6}, Corollary~\ref{cor4.7}, Lemma~\ref{lem4.8}, and H\"older's inequality that 
\begin{align*}
\|u(t)\|_{W_p^\xi(\Omega)} & \le b + b \int_0^t (t-s)^{-(1+\xi)/2} e^{-(t-s)/2} \Big( |\Omega|^{(3-p)/3} \|v(s)\|_{L_3(\Omega)} \\ 
& \hspace{5cm} + \| u(s)\|_{L_{3p/(3-p)}(\Omega)} \| \nabla w(s)\|_{L_3(\Omega)} \Big)\ \mathrm{d}s \\
& \le b + b \left[ b_3 + b_6 b_7(3p/(3-p)) \right] \int_0^t s^{-(1+\xi)/2} e^{-s/2}\ \mathrm{d}s \le b.
\end{align*}
Since $\xi>2/p$, the space $W_p^\xi(\Omega)$ is continuously embedded in $L_\infty(\Omega)$ and we deduce from the above estimate that
\begin{equation*}
\| u(t)\|_{L_\infty (\Omega)} \le b_9, \qquad t\ge 0.
\end{equation*}
Moreover, from \eqref{a2PL} and the comparison principle, we obtain
\begin{equation*}
\| v(t)\|_{L_\infty (\Omega)} \le \max \left\{ \frac{b_9}{\theta}, \| v_0 \|_{L_\infty (\Omega)} \right\}, \qquad t\ge 0,
\end{equation*}
which, together with Corollary~\ref{cor4.7} and the continuous embedding of $W_3^1(\Omega)$ in $L_\infty(\Omega)$, completes the proof of part~(a).

\medskip

\noindent (b) As $w$ is radially symmetric, an improved version of Proposition~\ref{prop4.3} is valid. Namely, in view of \cite[Theorem~2.1]{NSY1997}, for any $\eta >0$, there is $K(\eta) >0$ 
		  depending only on $\eta$ and $\Omega$ such that
		  $$\int_\Omega e^{w} \ \mathrm{d}x \le K(\eta) \, \exp \left( \left(\eta + \frac{1}{16\pi} \right)\| \nabla w\|_{L_2(\Omega)}^2 + \frac{2\| w\|_{L_1(\Omega)}}{|\Omega|} \right).$$
		We then proceed as in the derivation of \eqref{e4.4.5} with the help of this estimate with $\eta = (8\pi(1+\theta)D -M)/(32\pi M)$ and \eqref{e4.1.3} to deduce that
			\begin{align*}
	  \ml (u,v,w) & \ge -M  \ln \left( \|e^w\|_{L_1(\Omega)} \right) + L\left( \|u\|_{L_1(\Omega)} \right) 
		+ L_\theta\left( \|v\|_{L_1(\Omega)} \right) + (|\Omega|-1) \left( 1+ \frac{1}{\theta} \right)\\
		& \quad + \frac{1+\theta}{2} \left( D \|\nabla w\|_{L_2 (\Omega)}^2 + \alpha \|w\|_{L_2 (\Omega)}^2 \right) + \frac{1}{2} \|D\Delta w - \alpha w + v\|_{L_2 (\Omega)}^2  \\
		& \ge \frac{8\pi(1+\theta)D -(1+ 16\pi \eta) M}{16\pi} \|\nabla w\|_{L_2 (\Omega)}^2 + \frac{\alpha(1+\theta)}{2} \|w\|_{L_2 (\Omega)}^2  \\
		& \quad + \frac{1}{2} \|D\Delta w - \alpha w + v\|_{L_2 (\Omega)}^2 - M \ln (K(\eta)) - \frac{2M}{|\Omega|} \|w\|_{L_1 (\Omega)} + (|\Omega|-1) \left( 1+ \frac{1}{\theta} \right) \\
		&\ge \frac{8\pi(1+\theta)D -M}{32\pi} \|\nabla w\|_{L_2 (\Omega)}^2 + \frac{\alpha(1+\theta)}{2} \|w\|_{L_2 (\Omega)}^2 + \frac{1}{2} \|D\Delta w - \alpha w + v\|_{L_2 (\Omega)}^2 - b\ .
	\end{align*}
	Using this improved version of \eqref{e4.4.2} in the remaining part of Section~\ref{sec4}, we finish the proof of part~(b). 
%	\end{enumerate}
\end{proof}

%%%%%%%%%%%%%%%%
%%%%%%%%%%%%%%%%
\section{Unbounded solutions for large mass}\label{sec5}
%%%%%%%%%%%%%%%%
%%%%%%%%%%%%%%%%

Given $M>0$, we denote by $\ms_M$ the set of nonnegative stationary solutions $(u_\ast, v_\ast,w_\ast) \in W^2_{2, \mb} (\Omega; \RR^3)$ to \eqref{aPL} satisfying $\| u_\ast +v_\ast \|_{L_1 (\Omega)} = M$. In view of \eqref{a2PL}, this requires $u_\ast = \theta v_\ast$, which implies, together with \eqref{a1PL}, that
$$u_\ast = c \frac{e^{w_\ast}}{\| e^{w_\ast}\|_{L_1(\Omega)}}$$ for some $c > 0$, which is determined by the mass constraint. Hence, we define $\ms_M$ in the following way:

$(u_\ast, v_\ast,w_\ast) \in \ms_M$ if
\begin{align}
  & (u_\ast, v_\ast,w_\ast) \in W^2_{2, \mb} (\Omega; \RR^3), \qquad u_\ast, v_\ast, w_\ast \ge 0 \mbox{ in } \Omega\ , \label{e5.0.1} \\
  & u_\ast = \frac{\theta M}{\theta +1} \frac{e^{w_\ast}}{\| e^{w_\ast}\|_{L_1(\Omega)}}\ , \qquad 
  v_\ast = \frac{M}{\theta +1} \frac{e^{w_\ast}}{\| e^{w_\ast}\|_{L_1(\Omega)}}\ , \label{e5.0.2} \\
  &-D \Delta w_\ast + \alpha w_\ast = \frac{M}{\theta +1} \frac{e^{w_\ast}}{\| e^{w_\ast}\|_{L_1(\Omega)}} \quad\mbox{ in } \Omega\ , \qquad 
  \nabla w_\ast \cdot \mathbf{n} =0 \quad\mbox{ on } \partial \Omega\ . \label{e5.0.3}
\end{align}

As in \cite{Ho2002, HW2001, SS2001}, we begin with a lower bound for the Liapunov function $\mathcal{L}$ on $\mathcal{S}_M$ for appropriate values of the mass $M$.
	
%%%%%%%%%%%%%%%%
\begin{proposition}\label{prop5.1}
  \begin{enumerate}
   \item[(a)] If $M \in (4\pi(1+\theta)D, \infty) \setminus \left(4\pi(1+\theta)D \NN \right)$, then 
     $$\mu_M := \inf\limits_{(u_\ast, v_\ast,w_\ast) \in \ms_M} \ml (u_\ast, v_\ast,w_\ast) > -\infty\ .$$
   \item[(b)] If $\Omega = B_R(0)$ for some $R>0$ and $M \in (8\pi(1+\theta)D, \infty)$, then 
     $$\mu_M := \inf\limits_{(u_\ast, v_\ast,w_\ast) \in \ms_{M, rad}} \ml (u_\ast, v_\ast,w_\ast) > -\infty\ ,$$
     where $\ms_{M, rad} := \{ (u_\ast, v_\ast,w_\ast) \in \ms_M \: : \: u_\ast, v_\ast, w_\ast \mbox{ are radially symmetric}\}$. 
  \end{enumerate}
\end{proposition}
%%%%%%%%%%%%%%%%

\begin{proof}
%  \begin{enumerate}  \item[(a)] 
\noindent (a) Let $(u_\ast, v_\ast,w_\ast) \in \ms_M$. Then, in view of \eqref{e5.0.2}, \eqref{e5.0.3}, and the mass constraint $\|u_\ast + v_\ast\|_{L_1(\Omega)} =M$, we deduce from \eqref{dPL} that 
     \begin{align}\label{e5.1.2}
        \ml (u_\ast, v_\ast,w_\ast) &= \int_\Omega \left( u_\ast \ln u_\ast - u_\ast +1 + v_\ast \ln(\theta v_\ast) - v_\ast + \frac{1}{\theta}- (u_\ast+v_\ast)w_\ast \right)\ \mathrm{d}x \nonumber \\
        & \quad + \frac{1+\theta}{2} \left( D \|\nabla w_\ast\|_{L_2 (\Omega)}^2 + \alpha \|w_\ast\|_{L_2 (\Omega)}^2 \right) + \frac{1}{2} \|D\Delta w_\ast- \alpha w_\ast + v_\ast\|_{L_2 (\Omega)}^2 \nonumber \\
        & = \int_\Omega \left( (u_\ast +v_\ast) \ln u_\ast -(u_\ast+v_\ast) w_\ast \right)\ \mathrm{d}x - M + |\Omega| \left( 1 + \frac{1}{\theta} \right) \nonumber \\
        & \quad + \frac{1+\theta}{2} \left( D \|\nabla w_\ast\|_{L_2 (\Omega)}^2 + \alpha \|w_\ast\|_{L_2 (\Omega)}^2 \right) \nonumber \\
        &= \int_\Omega (u_\ast +v_\ast) \left( \ln \left( \frac{\theta M}{\theta+1} \right) - \ln\left( \|e^{w_\ast}\|_{L_1(\Omega)} \right) \right)\ \mathrm{d}x -M + |\Omega| \left( 1+ \frac{1}{\theta} \right) \nonumber \\
        & \quad + \frac{1+\theta}{2} \left( D \|\nabla w_\ast\|_{L_2 (\Omega)}^2 + \alpha \|w_\ast\|_{L_2 (\Omega)}^2 \right) \nonumber \\
        &= M \ln \left( \frac{\theta M}{\theta +1} \right) - M \ln \left(\| e^{w_\ast}\|_{L_1(\Omega)} \right)
        -M + |\Omega| \left( 1+ \frac{1}{\theta} \right) \nonumber \\
        & \quad + \frac{1+\theta}{2} \left( D \|\nabla w_\ast\|_{L_2 (\Omega)}^2 + \alpha \|w_\ast\|_{L_2 (\Omega)}^2 \right) \ .
     \end{align}
     As \eqref{e5.0.1} and \eqref{e5.0.3} imply 
     $$w_\ast \ge 0 \quad\mbox{in } \Omega \qquad\mbox{ and }\qquad \| w_\ast\|_{L_1(\Omega)} = \frac{M}{\alpha(\theta +1)}\ ,$$
     we define 
     \begin{equation}\label{e5.1.3}
       W:= w_\ast - \frac{M}{\alpha(\theta +1)|\Omega|} = w_\ast - \frac{1}{|\Omega|} \int_\Omega w_\ast(x) \ \mathrm{d}x\ . 
     \end{equation}
     Rewriting \eqref{e5.1.2} in terms of $W$ leads to
     \begin{align}\label{e5.1.4}
       \ml (u_\ast, v_\ast,w_\ast) 
       &= \frac{1+\theta}{2} \left( D \|\nabla W\|_{L_2 (\Omega)}^2 + \alpha \|W\|_{L_2 (\Omega)}^2 
       + \alpha|\Omega| \frac{M^2}{\alpha^2(\theta +1)^2|\Omega|^2}\right) \nonumber \\
       & \quad - M \ln \left(\| e^W\|_{L_1(\Omega)} \right) - \frac{M^2}{\alpha(\theta +1)|\Omega|} +  M \ln \left( \frac{\theta M}{\theta +1} \right)
       \nonumber \\ & \quad -M + |\Omega| \left( 1+ \frac{1}{\theta} \right) \nonumber \\
       &= \frac{M}{|\Omega|} \mf (W) - M \ln \left( |\Omega|\right) 
       - \frac{M^2}{2\alpha(\theta +1)|\Omega|} +  M \ln \left( \frac{\theta M}{\theta +1} \right) \nonumber \\
       &\quad -M + |\Omega| \left( 1+ \frac{1}{\theta} \right)\ ,
     \end{align}
     where
     \begin{equation}\label{e5.1.5}
        \mf (W) := \frac{(1+\theta)|\Omega|}{2M} \left( D \|\nabla W\|_{L_2 (\Omega)}^2 + \alpha \|W\|_{L_2 (\Omega)}^2 \right) 
        - |\Omega| \ln \left( \frac{ \| e^W\|_{L_1(\Omega)}}{|\Omega|} \right).
     \end{equation}
     Moreover, in view of \eqref{e5.0.3} and \eqref{e5.1.3}, $W$ is a solution to
     \begin{equation}\label{e5.1.6}
       -D \Delta W + \alpha W = \frac{M}{(\theta +1)|\Omega|} \left( \frac{|\Omega| e^{W}}{\| e^{W}\|_{L_1(\Omega)}} -1 \right) \quad\mbox{ in } \Omega\ , \qquad 
        \nabla W\cdot \mathbf{n} =0 \quad\mbox{ on } \partial \Omega\ ,
     \end{equation}
     along with $\int_\Omega W(x) \ \mathrm{d}x =0$. Hence, due to \eqref{e5.1.5}, \eqref{e5.1.6} along with 
     $M \in (4\pi(1+\theta)D, \infty) \setminus \left(4\pi(1+\theta)D \NN \right)$, we are in a position to apply \cite[Lemma~3.5]{HW2001} and conclude that there
     exists $\mu \ge 0$ which does not depend on $W$ such that 
     \begin{equation}\label{e5.1.7}
        \mf(W) \ge -\mu\ .
     \end{equation}
     Combining the latter with \eqref{e5.1.4} completes the proof of assertion~(a).
    
\medskip
%   \item[(b)] 
\noindent (b) The proof is the same as that of assertion~(a), we only use \cite[Corollary~3.7 \& Remark~3.8]{HW2001} instead of \cite[Lemma~3.5]{HW2001} to deduce the 
     lower bound \eqref{e5.1.7} for any $M \in (8\pi(1+\theta)D, \infty)$.
%  \end{enumerate}
\end{proof}

As in \cite{Ho2002, HW2001, SS2001}, the next step is to show that $\mathcal{L}$ is not bounded from below on the set $\mathcal{I}_M$ of initial conditions defined in \eqref{in} as soon as $M$ exceeds a specific threshold value. The argument given below is however more involved, due to the additional positive term in $\mathcal{L}$.
	
%%%%%%%%%%%%%%%%
\begin{proposition}\label{prop5.2}
Let $M>0$.
	\begin{enumerate}
		\item[(a)] If $M \in (4\pi(1+\theta)D, \infty)$,then 
		$$\inf\limits_{(u, v,w) \in \mathcal{I}_{M}} \mathcal{L}(u,v,w) = -\infty.$$
		\item[(b)] If $\Omega = B_R(0)$ for some $R>0$ and $M \in (8\pi(1+\theta)D, \infty)$, then 
		$$\inf\limits_{(u, v,w) \in \mathcal{I}_{M,rad}} \mathcal{L}(u,v,w) = -\infty,$$
		where $\mathcal{I}_{M, rad} := \{ (u, v,w) \in \mathcal{I}_{M} \: : \: u, v, w \mbox{ are radially symmetric}\}$. 
	\end{enumerate}
\end{proposition}
%%%%%%%%%%%%%%%%

\begin{proof}

\medskip
 
\noindent (a) As $\partial \Omega$ is smooth, upon a translation and a rotation, we may assume without loss of generality that $0\in \partial\Omega$, $\mathbf{n}(0) = (0,-1)^T$, 
        and that there exist $a_0 \in (0,1)$ and $\zeta \in C^2([-a_0,a_0])$ such that we have 
  \begin{equation*}
    \Omega \cap B_{a_0}(0) = \{ x \in B_{a_0}(0) \: : \: x_2 > \zeta (x_1) \} \quad\mbox{and}\quad 
    \partial\Omega \cap B_{a_0}(0) = \{ x \in B_{a_0}(0) \: : \: x_2 = \zeta (x_1) \}.
  \end{equation*}
  We first claim that there is $\bar{\omega}\in C([0,a_0])$ such that
	\begin{subequations}
	\begin{equation}
	\bar{\omega}(0)=0, \quad \bar{\omega}\ge 0, \label{o1}
	\end{equation}
	and, for all $a\in (0,a_0)$,
	\begin{equation}
	\sigma_a \subset \Omega \cap B_a(0) \subset \Sigma_a, \label{o2}
	\end{equation}
	\end{subequations}
where
\begin{align*}
\sigma_a & := \left\{ x=(r\cos(\omega),r\sin(\omega))\ : \ r\in [0,a), \ \omega\in (\bar{\omega}(a),\pi-\bar{\omega}(a)) \right\}, \\
\Sigma_a & := \left\{ x=(r\cos(\omega),r\sin(\omega))\ : \ r\in [0,a), \ \omega\in [0,\pi+\bar{\omega}(a)) \cup (2\pi-\bar{\omega}(a),2\pi) \right\}.
\end{align*}
Indeed, in view of $0 \in \partial \Omega$ we have $\zeta(0)=0$, while $\mathbf{n}(0) = (0,-1)^T$ implies 
$\zeta^\prime (0) =0$. Hence, with $A := \| \zeta^{\prime \prime} \|_{C ([-a_0,a_0])} /2 \in [0, \infty)$, a Taylor expansion implies 
\begin{equation*}
  |\zeta(s)| = |\zeta(s)-\zeta(0) - \zeta^\prime (0) s| \le A s^2 \le A a |s| \qquad\mbox{for all } s \in [-a,a]
\end{equation*}
and any $a \in (0,a_0]$.
Combining these properties of $\Omega$ and $\zeta$, we deduce that 
\begin{equation*}
 \{ x \in B_a(0) \: : \: x_2 > Aa|x_1| \} \subset \Omega \cap B_a(0) \subset \{ x \in B_a(0) \: : \: x_2 > -Aa|x_1| \}.
\end{equation*}
Hence,  \eqref{o1} and \eqref{o2} are satisfied for any $a \in (0,a_0)$ with the continuous function 
$\bar{\omega} (a) := \arctan(A a)$, $a \in [0,a_0]$.

Next, for $\eta\in (0,1)$ and $x\in\Omega$, define 
\begin{equation*}
\xi_\eta(x) := 2 \ln\left( \frac{\eta}{\eta^2 + \pi |x|^2} \right) \;\;\text{ and }\;\; \Xi_\eta(x) := \xi_\eta(x) - \frac{1}{|\Omega|} \int_\Omega \xi_\eta(y)\ \mathrm{d}y.
\end{equation*}
Clearly, $\xi_\eta$ belongs to $W_3^2(\Omega)$ and, for $(x,\eta)\in \Omega\times (0,1)$, 
\begin{align}
\nabla\xi_\eta(x) & = - \frac{4\pi x}{\eta^2 + \pi |x|^2}, \nonumber \\ 
D^2\xi_\eta(x) & = - \frac{4\pi}{\eta^2 + \pi |x|^2} \mathrm{id} + \frac{8\pi^2}{(\eta^2 + \pi |x|^2)^2} x\otimes x, \nonumber \\
- \Delta\xi_\eta(x) & = \frac{8\pi\eta^2}{(\eta^2 + \pi |x|^2)^2} = 8\pi e^{\xi_\eta(x)}. \label{up1}
\end{align}
In view of \eqref{o1} and $M>4\pi (1+\theta) D > 4\pi D$, we next fix $a\in (0,a_0)$ and $\eta_0\in (0,1)$ sufficiently small such that
\begin{subequations}\label{oo345}
\begin{align}
& \bar{\omega}(a) < \max\left\{ \frac{\pi}{4} ,\frac{M-4\pi D}{32D} \right\}, \label{o3} \\
& \eta_0^2 + \pi a^2 < 1, \label{o4} \\
& \eta_0^2 < \frac{M-4\pi D}{32D |\Omega|} \pi a^4, \label{o5}
\end{align}
\end{subequations}
and derive additional estimates on $\xi_\eta$ for $\eta\in (0,\eta_0)$. First, by \eqref{o2},
\begin{align*}
\left\| e^{\xi_\eta}\right\|_{L_1(\Omega)} & = \int_{\Omega\cap B_a(0)} \frac{\eta^2}{(\eta^2 + \pi |x|^2)^2}\ \mathrm{d}x + \int_{\Omega\cap B_a(0)^c} \frac{\eta^2}{(\eta^2 + \pi |x|^2)^2}\ \mathrm{d}x \\
& \le \int_{\Sigma_a} \frac{\eta^2}{(\eta^2 + \pi |x|^2)^2}\ \mathrm{d}x + \int_{\Omega\cap B_a(0)^c} \frac{\eta^2}{(\eta^2 + \pi a^2)^2}\ \mathrm{d}x \\
& \le \frac{\pi+2\bar{\omega}(a)}{2\pi} \left[ - \frac{\eta^2}{\eta^2 + \pi r^2} \right]_{r=0}^{r=a} + \frac{\eta^2 |\Omega|}{\pi^2 a^4} \\
& \le \frac{1}{2} + \frac{\bar{\omega}(a)}{\pi} + \frac{\eta_0^2 |\Omega|}{\pi^2 a^4},
\end{align*}
and
\begin{align*}
\left\| e^{\xi_\eta}\right\|_{L_1(\Omega)} & \ge \int_{\Omega\cap B_a(0)} \frac{\eta^2}{(\eta^2 + \pi |x|^2)^2}\ \mathrm{d}x \ge \int_{\sigma_a} \frac{\eta^2}{(\eta^2 + \pi |x|^2)^2}\ \mathrm{d}x \\
& =  \frac{\pi-2\bar{\omega}(a)}{2\pi} \left[ - \frac{\eta^2}{\eta^2 + \pi r^2} \right]_{r=0}^{r=a} = \frac{\pi - 2 \bar{\omega}(a)}{2\pi} \frac{\pi a^2}{\eta^2+\pi a^2} \\
& \ge \frac{\pi - 2 \bar{\omega}(a)}{2\pi} \frac{\pi a^2}{\eta_0^2+\pi a^2}.
\end{align*}
Hence, using \eqref{oo345},
\begin{equation}
\frac{\pi a^2}{4} < \left\| e^{\xi_\eta}\right\|_{L_1(\Omega)} < \frac{1}{2} + \frac{M-4\pi D}{16\pi D} = \frac{M+4\pi D}{16\pi D}. \label{oo1}
\end{equation}

We next turn to $\Xi_\eta$ and first derive a lower bound for $\eta\in (0,\eta_0)$. To this end, we compute
\begin{align*}
I_\eta & := \frac{2}{|\Omega|} \int_\Omega \ln\left( \eta^2 + \pi |y|^2 \right)\ \mathrm{d}y \\
& = \frac{2}{|\Omega|} \int_{\Omega\cap B_a(0)} \ln\left( \eta^2 + \pi |y|^2 \right)\ \mathrm{d}y + \frac{2}{|\Omega|} \int_{\Omega\cap B_a(0)^c} \ln\left( \eta^2 + \pi |y|^2 \right)\ \mathrm{d}y \\
& \ge \frac{2}{|\Omega|} \int_{\Omega\cap B_a(0)} \ln\left( \eta^2 + \pi |y|^2 \right)\ \mathrm{d}y + \frac{2}{|\Omega|} \int_{\Omega\cap B_a(0)^c} \ln\left( \pi a^2 \right)\ \mathrm{d}y.
\end{align*}
Since $\eta^2+\pi |y|^2 \le \eta_0^2 + \pi a^2 < 1$ for $(\eta,y)\in (0,\eta_0)\times \Sigma_a$ by \eqref{o4}, we infer from \eqref{o2} and \eqref{o3} that
\begin{align*}
I_\eta & \ge \frac{2}{|\Omega|} \int_{\Sigma_a} \ln\left( \eta^2 + \pi |y|^2 \right)\ \mathrm{d}y + \frac{4}{|\Omega|} \int_{\Omega\cap B_a(0)^c} \ln{a}\ \mathrm{d}y \\
& \ge \frac{\pi + 2\bar{\omega}(a)}{\pi |\Omega|} \left[ (\eta^2+\pi r^2) \ln{(\eta^2+\pi r^2)} - \eta^2 - \pi r^2 +1 \right]_{r=0}^{r=a} + \frac{4}{|\Omega|} \int_\Omega \ln{a}\ \mathrm{d}y \\
& \ge \frac{\pi + 2\bar{\omega}(a)}{\pi |\Omega|} \left( - 2 \eta^2 \ln{\eta} + \eta^2 - 1 \right) - 4 |\ln{a}| \ge - \frac{\pi + 2\bar{\omega}(a)}{\pi |\Omega|} - 4 |\ln{a}| \\
& \ge - \frac{1}{|\Omega|} \left( 1 + \frac{M}{16\pi D} \right) - 4 |\ln{a}|.
\end{align*}
Consequently, for $x\in \Omega$,
\begin{equation}
\Xi_\eta(x) = -2 \ln\left( \eta^2 + \pi |x|^2 \right) + I_\eta \ge - \nu_1, \label{oo2}
\end{equation}
with $R:= \mathrm{diam}(\Omega)/2$ and
\begin{equation*}
 \nu_1 := 2 \ln\left( 1 + 4\pi R^2 \right) + \frac{1}{|\Omega|} \left( 1 + \frac{M}{16\pi D} \right) + 4 |\ln{a}|.
\end{equation*}
Finally,
\begin{align}
\|\Xi_\eta\|_{L_2(\Omega)}^2 & = \int_\Omega \left(  -2 \ln\left( \eta^2 + \pi |x|^2 \right) + I_\eta \right)^2\ \mathrm{d}x = 4 \int_\Omega  \left( \ln\left( \eta^2 + \pi |x|^2 \right) \right)^2\ \mathrm{d}x - |\Omega| I_\eta^2 \nonumber \\
& \le 4 \int_{B_{2R}(0)} \left( \ln\left( \eta^2 + \pi |x|^2 \right) \right)^2\ \mathrm{d}x = 8\pi \int_0^{2R} r \left( \ln\left( \eta^2 + \pi r^2 \right) \right)^2\ \mathrm{d}r\nonumber \\
& = 4 \left[ (\eta^2 + \pi r^2) \left( \ln\left( \eta^2 + \pi r^2 \right) \right)^2 - 2 (\eta^2 + \pi r^2) \ln\left( \eta^2 + \pi r^2 \right) + 2(\eta^2+\pi r^2) \right]_{r=0}^{r=2R} \nonumber \\
& \le \nu_2^2, \label{oo3}
\end{align}
where
\begin{equation*}
\nu_2^2 := 4 \left[ (1 + 4\pi R^2) \left( \ln\left( 1 + 4 \pi R^2 \right) \right)^2 - 2 (1 + 4 \pi R^2) \ln\left( 1 + 4 \pi R^2 \right) + 2(1+4 \pi R^2) \right].
\end{equation*}

Now, for $\eta\in (0,\eta_0)$, we set
\begin{subequations}\label{oo4}
\begin{equation}
u_\eta := U_\eta \frac{e^{\xi_\eta}}{\|e^{\xi_\eta}\|_{L_1(\Omega)}}, \quad v_\eta := V_\eta \frac{e^{\xi_\eta}}{\|e^{\xi_\eta}\|_{L_1(\Omega)}}, \quad w_\eta = \Xi_\eta + \nu_1, \label{oo4a}
\end{equation}
with
\begin{equation}
U_\eta := M - 8\pi D \| e^{\xi_\eta}\|_{L_1(\Omega)}, \quad V_\eta := 8\pi D \| e^{\xi_\eta}\|_{L_1(\Omega)}. \label{oo4b}
\end{equation}
\end{subequations} 
We first observe that \eqref{oo1} and the lower bound on $M$ guarantee that 
\begin{equation}
U_\eta = \|u_\eta\|_{L_1(\Omega)} \in \left[2\pi\theta D ,M \right] \;\;\text{ and }\;\; V_\eta = \|v_\eta\|_{L_1(\Omega)} \in \left[ 2\pi^2 a^2 D, \frac{2+\theta}{2(1+\theta)} M \right], \label{oo9}
\end{equation}
while $w_\eta\ge 0$ by \eqref{oo2}, so that the triple $(u_\eta,v_\eta,w_\eta)$ defined in \eqref{oo4} belongs to $\mathcal{I}_M$. Also, 
\begin{equation*}
u_\eta = U_\eta \frac{e^{w_\eta}}{\|e^{w_\eta}\|_{L_1(\Omega)}} \;\;\text{ and }\;\; v_\eta = V_\eta \frac{e^{w_\eta}}{\|e^{w_\eta}\|_{L_1(\Omega)}}.
\end{equation*}

Next, on the one hand, 
\begin{align}
\int_\Omega u_\eta \ln u_\eta\ \mathrm{d}x &= \ln(U_\eta) \int_\Omega u_\eta\ \mathrm{d}x + \int_\Omega u_\eta w_\eta\ \mathrm{d}x - \ln \left( \| e^{w_\eta}\|_{L_1(\Omega)} \right) \int_\Omega u_\eta\ \mathrm{d}x \nonumber \\
&=  \int_\Omega u_\eta w_\eta\ \mathrm{d}x + U_\eta \ln(U_\eta) - U_\eta\ln \left( \| e^{w_\eta}\|_{L_1(\Omega)} \right) \label{oo5}
\end{align}
and
\begin{align}
\int_\Omega v_\eta \ln( \theta v_\eta)\ \mathrm{d}x &= \ln(\theta V_\eta) \int_\Omega v_\eta\ \mathrm{d}x + \int_\Omega v_\eta w_\eta\ \mathrm{d}x - \ln \left( \| e^{w_\eta}\|_{L_1(\Omega)} \right) \int_\Omega v_\eta\ \mathrm{d}x \nonumber \\
&=  \int_\Omega v_\eta w_\eta\ \mathrm{d}x + V_\eta \ln(\theta V_\eta) - V_\eta\ln \left( \| e^{w_\eta}\|_{L_1(\Omega)} \right). \label{oo6}
\end{align}
On the other hand, by \eqref{up1}, \eqref{oo3}, and \eqref{oo4},
\begin{align}
\|D\Delta w_\eta- \alpha w_\eta + v_\eta\|_{L_2(\Omega)}^2 & = \left\| D\Delta \xi_\eta + V_\eta \frac{e^{\xi_\eta}}{\left\| e^{\xi_\eta} \right\|_{L_1(\Omega)}} - \alpha \Xi_\eta - \alpha \nu_1 \right\|_{L_2(\Omega)}^2 \nonumber \\
& = \left\| D\Delta \xi_\eta + 8\pi D e^{\xi_\eta}  - \alpha \Xi_\eta - \alpha \nu_1 \right\|_{L_2(\Omega)}^2 \nonumber \\
& = \alpha^2 \left\|  \Xi_\eta + \nu_1 \right\|_{L_2(\Omega)}^2  \le 2 \alpha^2 \left( \left\|  \Xi_\eta \right\|_{L_2(\Omega)}^2 + |\Omega| \nu_1^2 \right) \nonumber \\
& \le 2\nu_3^2 := 2 \alpha^2 \left( \nu_2^2 + |\Omega| \nu_1^2 \right). \label{oo7}
\end{align}
We then infer from \eqref{oo5}, \eqref{oo6}, and \eqref{oo7} that
\begin{align}
\mathcal{L}(u_\eta, v_\eta,w_\eta) & = \int_\Omega \left( u_\eta \ln u_\eta - u_\eta +1 + v_\eta \ln(\theta v_\eta) - v_\eta + \frac{1}{\theta}- (u_\eta+v_\eta)w_\eta \right) \nonumber \\
& \quad + \frac{1+\theta}{2} \left( D \|\nabla w_\eta\|_{L_2(\Omega)}^2 + \alpha \|w_\eta\|_{L_2(\Omega)}^2 \right) 
+ \frac{1}{2} \|D\Delta w_\eta- \alpha w_\eta + v_\eta\|_{L_2(\Omega)}^2 \nonumber \\
& \le U_\eta \ln(U_\eta) - U_\eta +|\Omega| + V_\eta \ln(\theta V_\eta) - V_\eta + \frac{|\Omega|}{\theta} - (U_\eta + V_\eta)\ln \left( \| e^{w_\eta}\|_{L_1(\Omega)} \right) \nonumber \\
& \quad + \frac{1+\theta}{2} \left( D \|\nabla \Xi_\eta\|_{L_2(\Omega)}^2 + \alpha \|\Xi_\eta+\nu_1\|_{L_2(\Omega)}^2 \right) + \nu_3^2 \nonumber \\
& = U_\eta \ln(U_\eta) - U_\eta +|\Omega| + V_\eta \ln(\theta V_\eta) - V_\eta + \frac{|\Omega|}{\theta} - M \ln \left( \| e^{\Xi_\eta}\|_{L_1(\Omega)} \right) - M \nu_1 \nonumber \\
& \quad + \frac{1+\theta}{2} \left( D \|\nabla \Xi_\eta\|_{L_2(\Omega)}^2 + \alpha \|\Xi_\eta\|_{L_2(\Omega)}^2  + \alpha |\Omega| \nu_1^2 \right)  + \nu_3^2 \nonumber \\
& = \frac{M}{|\Omega|} \mathcal{F}(\Xi_\eta) + \mathcal{R}_\eta, \label{oo8}
\end{align}
where
\begin{equation*}
\mathcal{F}(\Xi_\eta) = \frac{(1+\theta)|\Omega|}{2M} \left( D \|\nabla \Xi_\eta\|_{L_2(\Omega)}^2 + \alpha \|\Xi_\eta\|_{L_2(\Omega)}^2 \right) - |\Omega| \ln \left( \frac{ \| e^{\Xi_\eta}\|_{L_1(\Omega)}}{|\Omega|} \right),
\end{equation*}
see \eqref{e5.1.5}, and
\begin{align*}
\mathcal{R}_\eta & := - M \ln(|\Omega|) + U_\eta \ln(U_\eta) - U_\eta +|\Omega| + V_\eta \ln(\theta V_\eta) - V_\eta + \frac{|\Omega|}{\theta} \\
& \quad - M \nu_1 + \frac{(1+\theta) \alpha |\Omega|}{2} \nu_1^2  + \nu_3^2.
\end{align*}
According to \cite[Section~3]{HW2001}, 
\begin{equation*}
\lim_{\eta\to 0} \mathcal{F}(\Xi_\eta) = -\infty,
\end{equation*}
while \eqref{oo9} ensures that $\sup_{\eta\in (0,\eta_0)} \mathcal{R}_\eta<\infty$. In view of these properties, it readily follows from \eqref{oo8} that $\mathcal{L}(u_\eta,v_\eta,w_\eta)\to -\infty$ as $\eta\to 0$, as claimed.

\medskip

\noindent (b) We recall that $\Omega=B_R(0)$ in that case. As above, for $\eta\in (0,1)$ and $x\in\Omega$, we define 
\begin{equation*}
\xi_\eta(x) := 2 \ln\left( \frac{\eta}{\eta^2 + \pi |x|^2} \right)  \;\;\text{ and }\;\; \Xi_\eta(x) := \xi_\eta(x)  - \frac{1}{|\Omega|} \int_\Omega \xi_\eta(y)\ \mathrm{d}y.
\end{equation*}
Then $\xi_\eta\in W_3^2(\Omega)$ and 
\begin{align*}
\left\| e^{\xi_\eta} \right\|_{L_1(\Omega)} & = \int_\Omega \frac{\eta^2}{(\eta^2 + \pi |x|^2)^2} \ \mathrm{d}x = \eta^2 \int_0^R \frac{2\pi r}{(\eta^2 + \pi r^2)^2} \ \mathrm{d}r \nonumber \\
& = - \left[ \frac{\eta^2}{\eta^2+\pi r^2}  \right]_{r=0}^{r=R} = \frac{\pi R^2}{\eta^2+\pi R^2}.
\end{align*}
Hence,
\begin{equation}
\frac{\pi R^2}{1 + \pi R^2} \le \left\| e^{\xi_\eta} \right\|_{L_1(\Omega)} \le 1. \label{up4}
\end{equation}
Furthermore,
\begin{align}
I_\eta & := \frac{2}{\pi R^2} \int_\Omega \ln\left( \eta^2 + \pi |y|^2 \right)\ \mathrm{d}y = \frac{4}{R^2} \int_0^R \ln\left( \eta^2 + \pi r^2 \right) r\ \mathrm{d}r \nonumber \\
& = \frac{2}{\pi R^2} \left[ \left( \eta^2 + \pi r^2 \right) \ln\left( \eta^2 + \pi r^2 \right) - \left( \eta^2 + \pi r^2 \right) + 1 \right]_{r=0}^{r=R} \nonumber \\
& \ge \frac{2}{\pi R^2} \left( -2 \eta^2 \ln{\eta} + \eta^2 - 1 \right) \ge - \frac{2}{\pi R^2} , \label{up2}
\end{align}
so that
\begin{equation}
\Xi_\eta(x) \ge - 2 \ln\left( \eta^2 + \pi R^2 \right) + I_\eta \ge - \nu_4 := - 2 \ln\left( 1 + \pi R^2 \right) - \frac{2}{\pi R^2}. \label{up3}
\end{equation}
Finally, 
\begin{align}
\|\Xi_\eta\|_{L_2(\Omega)}^2 & = \int_\Omega \left[ - 2 \ln\left( \eta^2 + \pi |x|^2 \right) + I_\eta \right]^2\ \mathrm{d}x = 4 \int_\Omega \left( \ln\left( \eta^2 + \pi |x|^2 \right) \right)^2\ \mathrm dx - |\Omega| I_\eta^2 \nonumber \\
& \le 4 \left[ \left( \eta^2 + \pi r^2 \right) \left[\ln\left( \eta^2 + \pi r^2 \right)\right]^2 - 2 \left( \eta^2 + \pi r^2 \right) \ln\left( \eta^2 + \pi r^2 \right) + 2\left( \eta^2 + \pi r^2 \right) \right]_{r=0}^{r=R} \nonumber \\
& \le \nu_5^2, \label{up5}
\end{align}
where
\begin{equation*}
\nu_5^2 := 4  \left[ \left( 1^2 + \pi R^2 \right) \left[\ln\left( 1 + \pi R^2 \right)\right]^2 - 2 \left( 1 + \pi R^2 \right) \ln\left( 1 + \pi R^2 \right) + 2\left( 1 + \pi R^2 \right) \right].
\end{equation*}

Now, as above, we set 
\begin{subequations}\label{up6}
\begin{equation}
u_\eta := U_\eta \frac{e^{\xi_\eta}}{\|e^{\xi_\eta}\|_{L_1(\Omega)}}\ , \quad v_\eta := V_\eta \frac{e^{\xi_\eta}}{\|e^{\xi_\eta}\|_{L_1(\Omega)}}\ , \ w_\eta = \Xi_\eta + \nu_4, \label{up6a}
\end{equation}
with
\begin{equation}
U_\eta := M - 8\pi D \| e^{\xi_\eta}\|_{L_1(\Omega)}, \quad V_\eta := 8\pi D \| e^{\xi_\eta}\|_{L_1(\Omega)}, \label{up6b}
\end{equation}
\end{subequations}
and deduce from \eqref{up4} and the lower bound on $M$ that
\begin{equation}
U_\eta\in \left[ \frac{\theta M}{1+\theta}, M \right]\;\;\text{ and }\;\; V_\eta\in \left[ \frac{8 \pi^2 R^2 D}{1+\pi R^2}, \frac{M}{1+\theta} \right]. \label{up100}
\end{equation}
In particular, owing to \eqref{up6}, \eqref{up100}, and the regularity of $\xi_\eta$, the triple $(u_\eta,v_\eta,w_\eta)$ belongs to $\mathcal{I}_{M,rad}$ for all $\eta\in (0,1)$. We next compute $\mathcal{L}(u_\eta, v_\eta,w_\eta)$ as in the previous case and argue as in the proof of \eqref{oo8} with the help of \eqref{up1}, \eqref{up3}, and \eqref{up5} to obtain that
\begin{equation}
\mathcal{L}(u_\eta, v_\eta,w_\eta) \le \frac{M}{|\Omega|} \mathcal{F}(\Xi_\eta) + \mathcal{R}_\eta \ , \label{up10}
\end{equation}
where
\begin{equation*}
\mathcal{F}(\Xi_\eta) = \frac{(1+\theta)|\Omega|}{2M} \left( D \|\nabla \Xi_\eta\|_{L_2(\Omega)}^2 + \alpha \|\Xi_\eta\|_{L_2(\Omega)}^2 \right) - |\Omega| \ln \left( \frac{ \| e^{\Xi_\eta}\|_{L_1(\Omega)}}{|\Omega|} \right)
\end{equation*}
as before and
\begin{align*}
\mathcal{R}_\eta & := - M \ln(|\Omega|) + U_\eta \ln(U_\eta) - U_\eta +|\Omega| + V_\eta \ln(\theta V_\eta) - V_\eta + \frac{|\Omega|}{\theta} \\
& \quad - M \nu_4 + \frac{(1+\theta)\alpha |\Omega|}{2} \nu_4^2 + \alpha^2 \left( \nu_5^2 + |\Omega| \nu_4^2 \right).
\end{align*}
According to \cite[Lemma~2]{Ho2002}, 
\begin{equation*}
\lim_{\eta\to 0} \mathcal{F}(\Xi_\eta) = - \infty\ ,
\end{equation*}
while \eqref{up100} implies that $\sup_{\eta\in (0,1)} \mathcal{R}_\eta < \infty$. It then readily follows from \eqref{up10} and the above properties that $\mathcal{L}(u_\eta,v_\eta,w_\eta) \longrightarrow - \infty$ as $\eta\to 0$ and the proof of (b) is complete.
\end{proof}

\begin{proof}[Proof of Theorem~\ref{theo1.3}]
\noindent (a) Gathering the outcome of Proposition~\ref{prop2.2},  Proposition~\ref{prop5.1}~(a) and Proposition~\ref{prop5.2}~(a), we argue as in \cite{Ho2002, HW2001, La2019, SS2001}, see also \cite{IY2013}, to conclude that, for $M\in (4\pi(1+\theta)D,\infty)\setminus (4\pi (1+\theta) D \mathbb{N})$, there are initial conditions in $\mathcal{I}_M$ for which the first component of the corresponding solution to \eqref{aPL}-\eqref{iPL} cannot be bounded in $L_\infty((0,\infty)\times \Omega)$ and thus infringes \eqref{apbu}.

\medskip

\noindent (b) In that case, combining Proposition~\ref{prop2.2},  Proposition~\ref{prop5.1}~(b) and Proposition~\ref{prop5.2}~(b) with the above mentioned references leads to the claim.
\end{proof}

%%%%%%%%%%%%%%%%
%%%%%%%%%%%%%%%%
%\section*{Acknowledgments}
%%%%%%%%%%%%%%%%
%%%%%%%%%%%%%%%%

%%%%%%%%%%%%%%%%
%%%%%%%%%%%%%%%%
\bibliographystyle{siam}
\bibliography{MPB}
%%%%%%%%%%%%%%%%
%%%%%%%%%%%%%%%%

\end{document}